\documentclass[11pt]{amsart}
\usepackage{amssymb,amscd,color,tikz}
\usepackage[all]{xy}
\usepackage[symbol]{footmisc}

\usepackage[
urlcolor=blue,
colorlinks=true,
linkcolor=blue,
citecolor=blue,
]{hyperref}

\newtheorem{thm}{Theorem}[section]

\newtheorem{lem}[thm]{Lemma}

\newtheorem{prop}[thm]{Proposition}
  
\newtheorem{cor}[thm]{Corollary}
\theoremstyle{definition}
\newtheorem{defn}[thm]{Definition}
\newtheorem{nota}[thm]{Notation}
\newtheorem{rmk}[thm]{Remark}
\newtheorem{ex}[thm]{Example}
\theoremstyle{remark}

\numberwithin{equation}{section}

\newcommand{\alb}{\mathrm{Alb}}

\newcommand{\op}{\mathrm}

\newcommand{\gb}{\mathfrak b}

\newcommand{\bz}{\mathbb{Z}}
\newcommand{\bc}{\mathbb{C}}

\newcommand{\bp}{\mathbb{P}}
\newcommand{\bq}{\mathbb{Q}}
\newcommand{\br}{\mathbb{R}}

\newcommand{\au}{\mathrm{Aut}_0}
\newcommand{\aut}{\mathrm{Aut}}

 \newcommand{\diff}{\mathrm{Diff}}

\newcommand{\id}{\mathrm{id}}

\newcommand{\num}{\mathrm{num}}

\newcommand{\proj}{\mathrm{Proj}}

\newcommand{\red}{\op{red}}
\newcommand{\rest}[1]{{}_{\left|#1\right.}}

\newcommand{\sign}{\mathrm{Sign}}

\newcommand{\sing}{\mathrm{sing}}

\newcommand{\tr}{\mathrm{tr}}

\newcommand{\mf}{\mathcal{F}}

\newcommand{\mo}{\mathcal{O}}

\newcommand{\restr}[1]{{\raisebox{-0.0\height}{$\mid_{#1}$}}}

\newcommand{\separate}{\medskip}

\begin{document}
\title[Automorphisms of surfaces  with $\MakeLowercase{q}=1$]{Automorphisms of surfaces of general type with $\MakeLowercase{q}=1$ acting trivially in cohomology}
\author{Jin-Xing Cai}
\address{Jin-Xing Cai\\LMAM, School of  Mathematical Sciences\\Peking University\\Beijing 100871\\P.R.~China}
\email{jxcai@math.pku.edu.cn}
\author{ Wenfei Liu}
\address{Wenfei Liu \\School of Mathematical Sciences\\Xiamen University\\Xiamen, Fujian 361005\\P.R.~China}
\email{wliu@xmu.edu.cn}
\thanks{This cooperation has been supported by the  NSFC (No.~11471020 and No.~11501012) and by the DFG via the Emmy Noether  Nachwuchsgruppe
``Modulr\"aume und Klassifikation von algebraischen Fl\"achen und Nilmannigfaltigkeiten mit linksinvarianter komplexer Struktur''. The second author was also partially supported by the Bielefelder Nachwuchsfonds.}
\date{}
\dedicatory{}

\begin{abstract}

Let $S$ be a complex minimal surface of general type with irregularity $q(S)=1$ and $\au(S)\subset \aut(S)$ the subgroup of automorphisms
    acting trivially on
$H^*(S, \mathbb{ Q})$. In this paper  we show that $|\au(S)|\leq4$,  and
if the equality holds then $S$ is a surface isogenous to a product of unmixed type.
Moreover, examples of surfaces with $|\au(S)|=4$ and all possible values of the geometric genus $p_g(S)$ are provided.
\end{abstract}
\maketitle

\tableofcontents

\section*{Introduction}

In studying the automorphism group $\aut(X)$ of a compact complex manifold $X$ it is important to consider its cohomology representation, that is, its natural action on the cohomology ring, say with rational coefficients. The action of the automorphism group on the cohomology is also relevant in the construction of \emph{fine} moduli spaces (\cite[Lecture~10]{Pop77}, see \cite{JL15} for a recent treatment of the case of hypersurfaces) and in the attempt to equip Teichm\"uller spaces with a complex structure (\cite[Sec.~1.3]{Cat13}). There the faithfulness of the action seems to be a desired property. We say $X$ is \emph{rationally cohomologically rigidified} if the action of $\aut(X)$ on $H^*(X,\bq)$ is faithful. Obviously, it is the same thing as requiring that $\aut(X)$ acts faithfully on $H^*(X,\bc)$, the cohomology ring with complex coefficients. In general, those automorphisms acting trivially on $H^*(X,\bq)$ are called \emph{numerically trivial} and they form a subgroup of the (full) automorphism group, to be denoted by $\au(X)$ in this paper.

It is well known that smooth projective curves are rationally cohomologically rigidified, unless the identity component of the automorphism group is nontrivial (in this case the genus is necessarily less than 2). The situation is more complicated in dimension two: there exist smooth projective surfaces with Kodaira dimension ranging from 0 to  2, which have automorphisms, not belonging to the identity component, acting trivially on the cohomology with rational coefficients (see \cite{Pe80}, \cite{BP83}, \cite{MN84}, \cite{Muk10} for surfaces of Kodaira dimension 0 and 1, and \cite{Cai06}, \cite{Cai07}, \cite{CLZ13} for surfaces of general type). 

The automorphism group of a surface $S$ of general type is finite, and thus $\au(S)$ does not lie in the identity component of $\aut(S)$ as soon as it is nontrivial. It turns out that nontrivial $\au(S)$ occurs only for those with irregularity $q(S)\leq 2$. Moreover, if $q(S)=2$ then the order of $\au(S)$ is at most 2, and in case of nontrivial $\au(S)$ the signature of the minimal model of $S$ vanishes (\cite[Theorem~1.1]{CLZ13}).

In this paper we investigate surfaces of general type with $q(S)=1$:
\begin{thm}\label{main} 
Let $S$ be a  complex minimal projective surface of general type with $q(S)=1$. Then we have $|\au(S)|\leq4$ with equality only if $S$ is a surface isogenous to a product of unmixed type.
\end{thm}

Surfaces isogenous to a product are those surfaces admitting a product of two smooth curves as an \'etale cover. By taking the Galois closure of the covering (\cite[Prop.~3.11]{Cat00}) we may give a more restrictive definition of them, see Definition~\ref{def: isog}.

In view of the diversity of surfaces of general type the geometric characterization for surfaces with $\au(S)$ of maximal order in Theorem~\ref{main} seems quite satisfactory. Inspired by the results of the current paper, the second named author \cite{Liu15} has shown recently that surfaces of general type with $q(S)=2$ and nontrivial $\au(S)$ must be isogenous to a product of curves. 

The bound in Theorem~\ref{main} is optimal, as series of surfaces with $q(S)=1$ and $|\au(S)|=4$ are constructed in Section~\ref{sec: examples}, realizing all possible values of the geometric genus. To complete the picture further we also povide examples of surfaces of general type with $q(S)=1$ and $\au(S)\cong\bz/3\bz$. See also \cite{Cai06} and \cite{Cai07} for examples of surfaces of general type with $q(S)=1$ and $\au(S)\cong\bz/2\bz$.  

One might ask for a simple reason for the existence of nontrivial $\au(S)$. In fact, a diffeomorphism that is homotopic to the identity map will act trivially on the cohomology, even with integral coefficients. In particular, if an automorphism of an algebraic surface $S$, viewed as a diffeomorphism of the underlying differential manifold of $S$, comes from the identity component $\diff^0(S)$ of the diffeomorphism group, then it acts trivially on the cohomology. Unfortunately, for those irregular surfaces with $|\au(S)|=4$ in Theorem~\ref{main} this is not the case, since surfaces isogenous to a product are rigidified, that is, we have $\aut(S)\cap\diff^0(S)=\{\id_S\}$, see \cite[Prop.~4.8]{CLZ13}.

Theorem~\ref{main} follows from Theorems~\ref{thm: bound k=2}, \ref{thm: G=4 k=2}, \ref{thm: bound k=1} and \ref{thm: isog k=1} where the genus of the Albanese fibration  of the surfaces with $q(S)=1$ and $|\au(S)|=4$ is also determined. The starting point of the proofs is that $\au(S)$ induces the trivial action on the Albanese variety, so that the Albanese map factors through the quotient map $S\rightarrow S/\au(S)$ (see Lemma~\ref{lem: alb quot}). 

The quotient map is of fundamental importance in studying automorphisms in general. In our context the quotient by $\au(S)$ inherits several invariants such as the geometric genus of the original surface $S$. This can be understood as giving certain bound on the quotient surface. Ultimately, we rely on the Bogomolov--Miyaoka--Yau inequality to conclude that $|\au(S)|\leq 4$. This bound has been obtained in \cite{Cai04} under the assumption that $\chi(\mo_S)>188$, where a large $\chi(\mo_S)$ is to ensure that the canonical map is well-behaved (cf.~\cite{Be79}). We focus instead on the canonical system rather than the map it induces and it is thus possible to deal with all  surfaces of general type with $q(S)=1$ in one go, regardless of their geometric genus.

The characterization of surfaces with maximal $\au(S)$ is divided into two steps.  First we prove the numerical equality $K_S^2=8\chi(\mo_S)$, which is equivalent to the vanishing of the signature of the underlying 4-dimensional differentiable manifold of $S$. A key role is played by certain versions of the equivariant signature formula of Hirzebruch and Zagier (\cite[p.~177]{HZ74}, see also \cite[1.6]{Cai09}). The second more subtle step consists in doing a careful analysis of the fixed loci,  which a priori are a collection of scattered points and curves, to show that every singular fibre of the Albanese map is of the form $2C$ where $C$ is a smooth curve. A lemma of Serrano \cite[Lemma~5]{Se95} then guarantees  that the surfaces are isogenous to a product of curves of unmixed type. 

\separate

\noindent{\bf Notation and Conventions.} We work over the complex numbers $\bc$.

Let $X$ be a compact complex manifold of dimension $n$. Then 
\begin{itemize}
 \item for a sheaf $\mf$ on $X$, $h^i(X,\mf)$ is the dimension of its $i$-th cohomology group $H^i(X,\mf)$ and $\chi(\mf)$ the Euler characteristic;
 \item $q(X):=h^1(X,\mo_X)$ and $p_g(X):=h^0(X, K_X)$ are the irregularity and the geometric genus of $X$ respectively;
 \item $e(X)$ is the topological Euler characteristic;
 \item if $X$ is even dimensional, $\sign(X)$ denotes the signature of the intersection form on the middle cohomology $H^n(X,\br)$;
 \item the Albanese torus of $X$ is denoted by $\alb(X)$ and the Albanese map by $a_X\colon X\rightarrow\alb(X)$;
 \item  the group of holomorphic automorphisms acting trivially on the cohomology ring $H^*(X,\bq)$ will be denoted by $\au(X)$. For simplicity of notation we often write $G_0$ for $\au(X)$.
\end{itemize}

If $f\colon S\rightarrow B$ is a fibration from a smooth projective surface onto a smooth projective curve, then genus  of $f$, denoted by $g(f)$, means the genus of a general fibre.

The symbol $\sim$ (resp.~$\sim_\bq$) denotes (resp.~$\bq$-)linear equivalence between (resp.~$\bq$-)divisors while $\equiv$ denotes numerical equivalence.

For a finite group $G$ we will denote its order by $|G|$. If it acts on a set $X$ then the fixed point set of an element $\sigma\in G$ is denoted by \[                                                                                                        
  X^\sigma:=\{p\in X\mid \sigma(p)=p\}.                                                                                                                                                   \]

For a finite abelian group $G$ we denote by
$\widehat{G}$  the character group of $G$. For a representation $V$
of $G$ and a character $\chi\in \widehat{G}$ we write
\[
 V^\chi=\{v\in V \mid g(v)=\chi(g)v \ \textrm{for all}\ g\in G\}. 
\]
Note that $V^\chi$ is contained in the subspace of $V$ that is pointwise fixed by $\ker(\chi)$.

The dihedral group of order $n$ will be denoted by $D_n$ and quaternion group of order 8 by $Q_8$. 

\separate

\noindent{\bf Acknowledgements.} We would like to thank Fabrizio Catanese for his interest in our project and a referee for providing an alternative description of the curves in Examples~\ref{ex: 1} and \ref{ex: 2}.

\section{Basic properties of $\au(X)$}\label{sec: au}
Let $X$ be a smooth projective variety and $G\subset\aut(X)$ a finite group of automorphisms. The quotient map $\pi\colon X\rightarrow X/G$ plays a fundamental role in studying the action of $G$. We make several observations about it in the case when $G$ acts trivially on the cohomology.

We remark that the following Lemmata~\ref{lem: hol inv quot} and \ref{lem: alb quot}, Remark~\ref{rmk: 1} are valid in the more general context of compact K\"ahler manifolds.
\begin{lem}\label{lem: hol inv quot}
Let $X$ be a smooth projective variety and $G$ a finite group of automorphisms acting trivially on $H^*(X,\bc)$. Let $\lambda\colon Y\rightarrow X/G$ be a resolution of singularities. Then the following holds.
\begin{itemize}
 \item[$\mathrm{(i)}$] $ h^i(Y, \mo_Y) = h^i(X,\mo_X)$ for any $0\leq i \leq \dim X$. As a consequence, 
 \[
  q(Y)=q(X),\, p_g(Y) =p_g(X) \text{ and }\chi(\mo_Y)=\chi(\mo_X).
 \]
\item[]  \hspace*{\dimexpr\linewidth-\textwidth\relax}
\begin{minipage}[t]{\textwidth}
Now assume that $p_g(X)>0$. Let $\tilde\pi\colon \tilde X \xrightarrow{\rho} X\dashrightarrow Y$ be a morphism eleminating the indeterminacy of the induced rational map $X\dashrightarrow Y$.
\end{minipage}
\item[$\mathrm{(ii)}$] 	There is an equality of complete linear systems
\[
|K_{\tilde X}|=\tilde\pi^*|K_Y| + \tilde R,
\] 
where $\tilde R$ is the ramification divisor of $\tilde\pi$.
\item[$\mathrm{(iii)}$] Suppose $D\subset X$ is an irreducible subvariety of codimension 1, fixed by some nontrivial element of $G$. 
Then $D$ is contained in the base locus of the canonical system $|K_X|$.
\end{itemize}
\end{lem}
\begin{proof}
(i) Since $G$ acts trivially on
$H^*(X,\bc)$, so does it on the direct summands $H^i(X,\mathcal O_X)$ in the Hodge decomposition of $H^*(X,\bc)$. We have by the Grothendieck--Leray spectral sequence
\begin{equation}\label{eq: coh quot}
  H^i(X/G,\mathcal O_{X/G}) = H^i(X/G,\pi_*^G \mathcal O_X) = H^i(X,\mathcal O_X)^G = H^i(X,\mathcal O_X).
\end{equation}
where $\pi\colon X\rightarrow X/G$ is the quotient map and $\pi_*^G\mathcal O_X$ denotes the $G$-invariant part of $\pi_*\mathcal O_X$. Since $X/G$ has only quotient (hence rational) singularities the cohomology groups of $X/G$ and its resolution $Y$ are the same:
\[
 H^i(X/G,\mathcal O_{X/G}) = H^i (Y, \mo_Y).
\]
Together with \eqref{eq: coh quot} this yields the desired conclusion.

(ii) The pull-back map $\tilde\pi^*\colon H^0(Y, K_Y)\rightarrow H^0(\tilde X, K_{\tilde X})$ is an injective linear map of vector spaces. By (i), $\tilde\pi^*$ is in fact an
isomorphism. This proves (ii).

(iii) The subvariety $D$ is contained in the image of $\tilde R$ in $X$, which lies in the base locus of $|K_{\tilde X}|$ by (ii). Since $|K_X|=\rho_*|K_{\tilde X}|$, where the push-forward operator
$\rho_*$ of divisors is defined as in \cite[Section~I]{Be78}, the subvariety $D$ lies in the base locus of $|K_X|$.
\end{proof}
\begin{rmk}\label{rmk: 1}
By the same proof Lemma~\ref{lem: hol inv quot} is also valid if we only assume that $G$ acts trivially on $H^i(X,\mo_X)$ for $0\leq i\leq \dim X$.
\end{rmk}

For lack of an appropriate reference, we give a proof of the following version of the topological Lefschetz fixed point formula. 
\begin{thm}\label{thm: lefschetz}
Let $X$ be a compact differentiable manifold (resp.~a compact complex space), and let $\sigma\in\diff(X)$ (resp.~$\sigma\in\aut(X)$) be of finite order. Then one has the topological Lefschetz fixed point formula
\begin{equation}\label{eq: lefschetz}
e(X^\sigma)=\sum_{0\leq i \leq n}(-1)^i \tr \left(\sigma^*\,|\,H^i(X, \bc)\right)
\end{equation}
where $\tr(\cdot)$ denotes the trace of an endomorphism of a vector space.
\end{thm}
\begin{proof}
We first remark that $X$ has a finite $\langle\sigma\rangle$-equivariant triangulation. Indeed, if $X$ is a compact differentiable manifold, the existence of such a triangulation is guaranteed by \cite{I78}. If $X$ is a compact complex space, then $X/\langle\sigma\rangle$ is again a compact complex space and we can stratify it into locally closed strata $A_j$ such that all the points over a single $A_j$ have the same stabilizer of the $\langle\sigma\rangle$-action. By \cite[Corollary~2.2]{J83} there is a finite triangulation of $X/\langle\sigma\rangle$ such that each stratum $A_j$ is a union of the support of (open) simplices, and one obtains a finite $\langle\sigma\rangle$-equivariant triangulation on $X$ by \cite[Theorem~5.5]{I83}.
 
 Let $C^i(X)$ be the vector space of the $i$-cochains over $\bc$, with basis dual to the set of $i$-simplices. The action $\sigma^*$ on $C^i(X)$ is induced by the permutation of $\sigma$ on the set of $i$-simplices. Thus we have
\begin{equation}\label{eq: inv cochain}
\tr \left(\sigma^*\,|\,C^i(X)\right)=\dim C^i(X)^\sigma =\dim C^i(X^\sigma).
\end{equation}
where $C^i(X)^\sigma$ is the $\sigma^*$-fixed part of $C^i(X)$ and it coincides with space of $i$-cochains supported on $X^\sigma$.

On the other hand, $H^i(X, \bc)$ are the homology groups of the cochain complex
 \[
 0\rightarrow C^0(X) \rightarrow C^1(X) \rightarrow\cdots\rightarrow C^{n-1}(X)\rightarrow C^n(X)\rightarrow 0
 \]
 where $n$ is the dimension of $X$. It is not hard to see that
 \begin{align*}
  \sum_{0\leq i \leq n}(-1)^i \tr \left(\sigma^*\,|\,H^i(X, \bc)\right) &= \sum_{0\leq i \leq n}(-1)^i \tr \left(\sigma^*\,|\,C^i(X)\right) \\&=  \sum_{0\leq i \leq n}(-1)^i \dim C^i(X^\sigma)\hspace{1cm} \text{ by \eqref{eq: inv cochain}}\\
  &=e(X^\sigma)
 \end{align*}
\end{proof}

The topological and holomorphic Lefschetz fixed point formulae (see \cite[Theorem~4.6 and Proposition~4.8]{AS68} for the later) have the following consequence.
\begin{lem}(\cite[Lemma 2.1]{Cai12})\label{lem: inv} Let $S$ be a
complex  nonsingular   projective surface. If there is
  an involution $\sigma$ of $S$ which acts trivially in
$H^2(S, \bq)$, then $K_S^2=8\chi(\mo_S)+\sum_{i=1}^m D_i^2$,
where   $D_1,\ \cdots,\ D_m$ ($m\geq0$) are $\sigma$-fixed curves.
\end{lem}

\begin{lem}\label{lem: alb quot}
 Let $X$ be a smooth projective variety with topological Euler characteristic $e(X)\neq 0$ and $G$ a subgroup of $\au(X)$. Then the Albanese map of $X$ factors as
 \[
a_X\colon X\xrightarrow{\pi} X/G\rightarrow \alb(X)
  \]
where $\pi\colon X\rightarrow X/G$ is the quotient map.
\end{lem}
\begin{proof}
Let $\sigma \in G$. The automorphism $\sigma_a$ of $\alb(X)$, induced by $\sigma$, fits into the following commutative diagram:
\[
 \xymatrix{
X \ar[r]^\sigma \ar[d]_{a_X} & X \ar[d]^{a_X}\\
 \alb(X)\ar[r]^{\sigma_a} &\alb(X)
}
\]
Since $\sigma$ induces the trivial action on $H^1(X,\bc)$, which can be identified with $H^1(\alb(X),\bc)$, the induced map $\sigma_a$ must be a translation of $\alb(X)$.  

On the other hand, since  $e\left(X^\sigma\right) = e(X)\neq0$ by \eqref{eq: lefschetz}, the fixed point set $X^\sigma$ is not empty. Note that $\sigma_a$ fixes the point $a_X(p)$ for any $p\in X^\sigma$,  so it can only be the identity map.
\end{proof}

Let $Y$ be a smooth model of $X/G$. By the universality of the Albanese maps and Lemma~\ref{lem: alb quot} we know that the Albanese varieties $\alb(Y)$ and $\alb(X)$ can be identified after fixing suitable base points for the Albanese maps. Indeed, we have a commutative diagram 
\begin{equation}\label{diag: alb}
 \begin{aligned}
 \xymatrix{
   X\ar[r]^\pi & X/G\ar[r]\ar@{=}[d] &\alb(X)\ar@{=}[d]\\
  Y  \ar@{-->}[r] & X/G\ar[r] &\alb(Y).
   }
    \end{aligned}
   \end{equation}

From now on we focus on the case of surfaces.
\begin{lem}\label{lem: g alb}
Let $S$ be a surface of general type with $q(S)=1$ and $G$ a subgroup of $\au(S)$. Suppose $X$ is a smooth model of $S/G$. Then the Kodaira dimension $\kappa(X)\geq 1$ and the equality holds if and only if the Albanese map $a_X\colon X\rightarrow \alb(X)$ has genus 1.
\end{lem}
\begin{proof}
Recall that the automorphism groups of surfaces of general type are finite. By Lemma~\ref{lem: hol inv quot} we have 
 \[
  p_g(X)=p_g(S)\geq 1 \text{ and }q(X)=q(S)=1.
 \]
So $\kappa(X)\geq 1$ by the Kodaira--Enriques classification of algebraic surfaces. 
 
If $X$ is of general type then any fibration has genus at least $2$.

Now suppose $X$ has Kodaira dimension 1. We consider the $m$-canonical map $\varphi_m\colon X\rightarrow B$ of $X$ for a sufficiently large and divisible $m$. Then $\varphi_m$ is an elliptic fibration. Since $q(X)=1$, the genus of $B$ is at most $1$. If $g(B)=0$, then $q(X)= g(B) + q(\varphi_m^*b)$ and $X$ is birational to $B\times \varphi_m^*b$, where $\varphi_m^*b$ is the fiber over a general $b\in B$. This is a contradiction to the fact that $X$ has Kodaira dimension 1. So $g(B)=1$ and $\varphi_m$ is just the Albanese map of $X$. It follows that the genus of $a_X$ is 1.

This finishes the proof of the lemma.
\end{proof}

We end this section with a useful observation.
\begin{lem}\label{lem: neg}
Let $S$ be a smooth projective surface and $C\subset S$ an irreducible curve with negative self-intersection. Then the following holds.
\begin{enumerate}
\item The curve $C$ is invariant under the action of $\au(S)$. 
\item Suppose furthermore that $f\colon S\rightarrow B$ is a fibration preserved by an automorphism $\sigma\in\au(S)$, that is, $f\circ \sigma=f$, and $C$ is a section of $f$. Then $C$ is fixed pointwise by $\sigma$.
\end{enumerate}
\end{lem}
\begin{proof}
(i) Suppose on the contrary that $\gamma(C)\neq C$ for some $\gamma\in \au(S)$. Then $C^2=C\cdot\gamma(C)\geq0$, a contradiction to the assumption.

(ii)  For every $p\in C$, since $C$ is a section of $f\colon S\rightarrow B$, we have $\{p\}=C\cap F_b$ with $b=f(p)$. Due to (i) and the assumption, both $C$ and $F_b$ are preserved by $\sigma$,  so $p$ is $\sigma$-fixed.
\end{proof}

\section{Surfaces with quotient of general type}\label{sec: k=2}
Let $S$ be a minimal surface of general type with $q(S)=1$ and $G_0:=\au(S)$ the automorphism group acting trivially on $H^*(S,\bc)$. We know from Lemma~\ref{lem: g alb} that $\kappa(S/G_0)\geq 1$, where $\kappa(S/G_0)$ denotes the Kodaira dimension of a smooth model of $S/G_0$.

Let $\lambda\colon\tilde T\rightarrow S/G_0$ be the minimal resolution of singularity and $\eta\colon \tilde T\rightarrow T$ the contraction to the minimal model of $\tilde T$. Then we have the following commutative diagram
\begin{equation}\label{diag: resol}
\begin{aligned}
 \xymatrix{
 \tilde S \ar[r]^{\tilde\pi}\ar[d]_\rho &\tilde T\ar[d]^\lambda\ar[r]^\eta& T\\
  S \ar[r]^\pi & S/G_0&
 }
 \end{aligned}
\end{equation}
where $\pi\colon S\rightarrow S/G_0$ is the quotient map and $\tilde S$ is the minimal resolution of singularities of the normalization of the fibre product $S\times_{S/G} \tilde T$.

In this section we will treat the case where $T$ is of general type.

\subsection{Bounding $|\au(S)|$, part I}
We first bound the order of an automorphism group in terms of the volumes of the originial surface and the quotient, thus improving this kind  of results previously obtained by  Xiao (\cite[Lemma~2 and Prop.~1 (i)]{X94}).
\begin{prop}\label{prop: X94}
Let $S$ be a minimal surface of general type and $G$ a group of its automorphisms. Assume the quotient $S/G$ is of general type and let $T$ be its minimal (smooth) model. Then the following holds.
\begin{enumerate}
 \item[(i)] $K_S^2\geq |G|K_T^2 + \sum (r_C-1) K_S C$,
  where the sum is taken over all irreducible curves 
   $C\subset S$ and $r_C$ is the order of the stabilizer at a general point of $C$.  In particular, we have $K_S^2\geq |G|K_T^2$.
 \item[(ii)] $K_S^2 = |G|K_T^2$ if and only if the following holds:
\begin{enumerate}
 \item the fixed point set $S^\gamma$ is finite for any nontrivial $\gamma\in G$;
 \item the quotient $S/G$ has at most canonical singularities.
\end{enumerate}
In this case $T$ is isomorphic to the minimal resolution of singularities of $S/G$.
\end{enumerate}
\end{prop}
\begin{proof}Let $\pi\colon S\rightarrow S/G$ be the quotient map.
For any irreducible curve $C\subset S$ we denote by $\bar{C}$ the image of
$C$ under $\pi$. There is a $\bq$-linear equivalence:
\begin{equation}\label{eq: quotient volume}
K_S \sim_\bq \pi^* \left(K_{S/G} + \sum_C\left(1-\frac{1}{r_C}\right) \bar C\right),
\end{equation}
where the sum is taken over all irreducible curves $C\subset S$ and $r_C$ is the order of the stabilizer at a general point of $C$.  Since $K_S$ is nef and $\pi$ is finite, the $\bq$-divisor
 $K_{S/G} + \sum\left(1-\frac{1}{r_C}\right) \bar C$ is also nef.

Let $\lambda\colon\tilde T\rightarrow S/G$ be the minimal resolution of
singularities and $\eta\colon \tilde T\rightarrow T$ the contraction
to the minimal model, see the following diagram:
\[
 \xymatrix{&&\tilde T\ar[ld]_\lambda\ar[rd]^\eta&\\
   S \ar[r]^\pi & S/G && T
 }
 \]
We have 
\[
\lambda^* K_{S/G}  \sim_\bq K_{\tilde T}+ \sum a_i E_i \text{ and } \eta^* K_T \sim_\bq K_{\tilde T} - A,
\]
where $E_i$ are the exceptional divisors of $\lambda$ and $A$ is an
effective divisor supported on the whole exceptional locus of
$\eta$. Since the resolution of singularities $\lambda\colon \tilde T\rightarrow S/G$ is minimal, we have $a_i\geq 0$ in the above
formula. It follows that
\begin{equation}\label{eq: quotient volume1}
\lambda^*\left( K_{S/G} + \sum\left(1-\frac{1}{r_C}\right) \bar C\right) \sim_\bq \eta^* K_T +\Omega,
\end{equation}
 where $\Omega:= \sum\left(1-\frac{1}{r_C}\right)\lambda^* \bar C+\sum a_i E_i +
A$ is an effective divisor.

Since both $\lambda^*\left( K_{S/G} + \sum\left(1-\frac{1}{r_C}\right) \bar C\right)$ and
$\eta^* K_T$ are big and nef, the inequality in the following computation holds:
\begin{equation}\label{xiaoineq}
\begin{aligned}
K_S ^2 &=|G|\left(K_{S/G} + \sum\left(1-\frac{1}{r_C}\right) \bar C\right)^2\ \
 \textrm{ (by \eqref{eq: quotient volume})} \\
&= |G|\lambda^*\left(K_{S/G} + \sum\left(1-\frac{1}{r_C}\right) \bar C\right)^2 \\
&= |G|\left((\eta^* K_T)^2 + \lambda^*\left(K_{S/G} + \sum\left(1-\frac{1}{r_C}\right) \bar C\right)\Omega  + \eta^* K_T\Omega\right)\\
&\geq |G| K_T^2 +|G| \left(K_{S/G} + \sum\left(1-\frac{1}{r_C}\right) \bar C\right)
\sum\left(1-\frac{1}{r_C}\right)\bar C\\
&=|G|K_T^2 + \pi^*  \left(K_{S/G} + \sum\left(1-\frac{1}{r_C}\right) \bar C\right)
\sum\left(1-\frac{1}{r_C}\right)\pi^*\bar C\\
&=|G|K_T^2 + K_S \sum (r_C-1)C.
\end{aligned}
\end{equation}
 This establishes (i).

Now by (\ref{xiaoineq}), $K_S^2 = |G|K_T^2$ holds if
and only if
\begin{align}\label{eq x1}
 \lambda^*\left( K_{S/G} + \sum\left(1-\frac{1}{r_C}\right) \bar C\right)
 \Omega=0 \ \ \textrm{and}\ \
 \eta^* K_T\Omega =
0,
\end{align}
which by the Hodge index theorem (assuming (\ref{eq x1}) we have $\Omega^2=0$ by
\eqref{eq: quotient volume1}) is in turn
equivalent to $\Omega=0$, that is, all of the three effective $\bq$-divisors $ \sum a_i E_i$, $\sum\left(1- \frac{1}{r_C}\right) \lambda^*\bar C$ and $A$ are 0. Hence the equality $K_S^2 = |G|K_T^2$ implies (a) and (b).

Conversely, (a) and (b) imply $ K_S\sim\pi^*K_{S/G} \text{ and } K_{\tilde T}\sim\lambda^*K_{S/G}$.
The nefness of $K_S$ implies the nefness of $K_{S/G}$ and in turn that of $K_{\tilde T}$. So $\tilde T$ is already minimal and we infer that
\[
 K_S^2 = |G|K_{S/G}^2 = |G|K_{\tilde T} = |G|K_T^2.
\]
\end{proof}

\begin{thm}\label{thm: bound k=2}
Let $S$ be a minimal surface of general type with $q(S)=1$ such that $\kappa(S/G_0)=2$, where $G_0:=\au(S)$. Then $|G_0|\leq 4$.
\end{thm}
\begin{proof}
Let $T$ be the minimal model of $S/G_0$. By Lemma~\ref{lem: hol inv quot},
\[
q(T)=q(S)>0 \text{ and }\chi (\mo_T) =\chi (\mo_S)
\]
 and hence one has by \cite[Lemma 14]{Bo}
 \begin{equation}\label{eq: noe irr}
  K_T^2 \geq2\chi (\mo_T) =2 \chi (\mo_S).
 \end{equation}
 Combined with Proposition~\ref{prop: X94} we obtain
 \begin{align}\label{eq: bound k=2}
 K_S^2\geq|G_0|K_T^2 \geq 2|G_0|\chi (\mo_S).
  \end{align}
The theorem follows from \eqref{eq: bound k=2} together with the Bogomolov--Miyaoka--Yau inequality.
\end{proof}

\subsection{Surfaces with $|\au(S)|=4$, part I}
Now we investigate the surfaces with $\au(S)$ of maximal order and characterize them as follows.
\begin{thm}\label{thm: G=4 k=2}
Let $S$ be a minimal surface of general type with $q(S)=1$ such that $\kappa(S/G_0)=2$ and $|G_0|=4$ where $G_0:=\au(S)$. Then the following holds.
 \begin{enumerate}
\item[(i)] The Albanese fibration $a_S\colon S\rightarrow \alb(S)$ has genus 5.
\item[(ii)] The group $G_0$ is isomorphic to $(\bz/2\bz)^2$.
\item[(iii)] $S$ is a surface isogenous to a product of unmixed type.
 \end{enumerate}
\end{thm}

We recall the definition of surfaces isogenous to a (higher) product.
\begin{defn}\label{def: isog} Let $C,\,D$ be smooth curves of genus at least 2, and $G$ is a finite group acting (faithfully) on $C$ and $D$. If the diagonal subgroup $\Delta_G$ of $G\times G$ acts freely on $C\times D$ then the smooth quotient $S:=(C\times D)/\Delta_G$ is called a surface isogenous to a product of unmixed type.
\end{defn}
\begin{rmk}
 There is also the notion of surfaces isogenous to a product of \emph{mixed} type $S=(C\times C)/G$ where the group $G$ acts freely and interchanges the two factors of $C\times C$.
\end{rmk}

Given a surface isogenous to a product of unmixed type $S=(C\times D)/\Delta_G$ its invariants satisfy
\[
 K_S^2 = 8\chi(\mo_S) = \frac{8}{|G|}(g(C)-1)(g(D)-1)\text{ and } q(S)=g(C/G)+g(D/G).
\]

We need two intermediate results for the proof of Theorem~\ref{thm: G=4 k=2}.
\begin{lem}\label{lem: iso}
Let $S$ be as in Theorem~\ref{thm: G=4 k=2}. Then the set $S^\sigma$ is finite for any nontrivial $\sigma\in \au(S)$.
   \end{lem}
 \begin{proof}
We use a similar argument as in \cite[3.6]{CLZ13}.

Write $G_0=\au(S)$ as before. Since $G_0$ is of order 4, it suffices to prove the lemma for $\sigma\in G_0$ that has order 2. 

Suppose on the contrary that such a  $\sigma$ fixes curves $D_1,\dots,D_u$. Then, by Lemma~\ref{lem: inv},
 \begin{equation}\label{top}
 K_S^2=8\chi(\mo_S)+\sum_{i=1}^u D_i^2 \hspace{1cm} 
 \end{equation}
On the other hand, we have by Proposition~\ref{prop: X94}
 \begin{align*}
K_{S}^2  &\geq 4K_T^2+\sum_{i=1}^u  K_{S}D_i\\
 &\geq8\chi(\mo_S)+(K_{S}-\sum_{i=1}^uD_i)\sum_{i=1}^uD_i+\sum_{i=1}^uD_i^2 \ \ \
 (\textrm{by \eqref{eq: noe irr}})\\
  &\geq 8\chi(\mathcal O_S)+2+\sum_{i=1}^uD_i^2,
 \end{align*}
 where $T$ is the minimal model of $S/G_0$ and the last inequality is because each $\sigma$-fixed curve is contained in the fixed part of $|K_S|$ (Lemma~\ref{lem: hol inv quot})  and each effective canonical divisor of $S$ is $2$-connected (cf.~\cite[VII, Prop.~6.2]{BHPV04}).
 This contradicts (\ref{top}).
 \end{proof}

\begin{cor}\label{cor: G=4 k=2}
Let $S$ be as in Theorem~\ref{thm: G=4 k=2} and $T$ the minimal model of $S/G_0$. Then the following holds.
\begin{enumerate}
\item[(i)] $K_S^2 = 8\chi(\mo_S)$ or, equivalently, $e(S)=4\chi(\mo_S)$.
\item[(ii)] The invariants of $T$ satisfy $K_T^2 = 2\chi(\mo_T)$ and the Albanese fibration $a_T\colon T\rightarrow \alb(T)$ has genus 2, with singular fibres all of type (0) in the sense of Horikawa \cite{Hor77}.
\item[(iii)] The quotient $S/G_0$ has only canonical singularities and $T$ is isomorphic to the minimal resolution of singularities of $S/G_0$.
\end{enumerate}
\end{cor}
\begin{proof}
(i) The first equality follows from Lemma~\ref{lem: iso} and Lemma~\ref{lem: inv}, applied to an involution from $G_0$. The second inequality is equivalent to the first one by the Noether formula.

(ii) Since $K_S^2=8\chi(\mo_S)$, the inequalities in \eqref{eq: noe irr} and \eqref{eq: bound k=2} are all equalities. In particular, $K_T^2=2\chi (\mo_T)$. By \cite[Theorem 5.1]{Hor81} the Albanese map $a_T\colon T\rightarrow \alb(T)$ is a fibration of  genus 2, whose singular fibres are all of type (0) in the sense of \cite{Hor77}.

(iii) Since $K_S^2=|G_0|K_T^2$ holds, the assertion follows from Proposition~\ref{prop: X94}.
\end{proof}

In the following lemma we record how $(\bz/2\bz)^2$-actions on surfaces behave when there are no fixed curves.
\begin{lem}\label{lem: Z22 acts}
Let $G\cong(\bz/2\bz)^2$ be a group of automorphisms of a smooth surface $S$. Denote by $\sigma_1,\sigma_2, \sigma_3$ be the three involutions in $G$. If $\sigma_1$ and $\sigma_2$ have only isolated singularities then the three fixed point sets
$S^{\sigma_i}$, $1\leq i\leq 3$ are pairwise disjoint.
\end{lem}
\begin{proof}Otherwise, there is a point $p\in S$
fixed by the whole group $G$. Then there are local coordinates $(x,y)$ around $p\in
S$ such that each $\gamma\in G$ acts as $$(x,y)\mapsto
(\chi_1(\gamma)x,\chi_2(\gamma)y),$$ where $\chi_1,\chi_2\in\widehat
G$ are two distinct characters of $G$. The assumption implies that
$\sigma_1$, $\sigma_2\not\in \ker(\chi_i)$ for $i=1, 2$. So
$\ker(\chi_i)=\langle \sigma_3 \rangle$ for $i=1, 2$, and hence
$\sigma_3$ induces the trivial action on the tangent space of $S$ at $p$. This implies that $\sigma_3$ is trivial, a  contradiction.
\end{proof}

\begin{proof}[Proof of Theorem~\ref{thm: G=4 k=2}]
Let $T$ be the minimal model of $S/G_0$ and $B$ the identified Albanese varieties  $\alb(S)$ and $\alb(T)$, see \eqref{diag: alb}.

(i) For a general $b\in B$ the fibre $a_T^*b$ is the quotient of $a_S^*b$ by $G_0$. By Lemma \ref{lem: iso} the action of $G_0$ on $a_S^*b$ is free and hence the quotient map $a_S^*b\rightarrow a_T^* b$ is \'etale. Since $g(a_T^*b)=2$ by Corollary~\ref{cor: G=4 k=2} one computes easily $g(a_S^*b) = 5$.

(ii) Since $|G_0|=4$, there are two possibilities: $G_0\cong\bz/4\bz$ or
$(\bz/2\bz)^2$. Suppose on the contrary that $G_0\cong \bz/4\bz$,
generated by $\alpha$. By Lemma~\ref{lem: iso}  both $\alpha$ and
$\alpha^2$ have only isolated singularities. It follows that any
fixed point of $\alpha$ is either of weight $\frac{1}{4}(1,1)$ or of
weight $\frac{1}{4}(1,3)$. Let $k_1$ (resp.~$k_3$) be the number of
$\alpha$-fixed points where the action is of weight
$\frac{1}{4}(1,1)$ (resp.~$\frac{1}{4}(1,3)$).

By Corollary~\ref{cor: G=4 k=2} the quotient $S/G_0$ has only canonical singularities, so $k_1=0$. By the topological Lefschetz fixed point formula
 \eqref{eq: lefschetz} 
 \begin{equation}\label{eq: number of fixed points}
 k_3= \#\,S^\alpha = e(S).
 \end{equation}

 Now we apply the equivariant signature formula to $\alpha$ (cf.~\cite[Equation (12), p.~177]{HZ74}, \cite[1.6]{Cai09}, or \cite[3.3]{Cai12}) and obtain in our case
 \begin{equation}\label{eq: equivariant signature alpha}
   4 \sign (S/\alpha) = \sign (S) +2k_3,
 \end{equation}
   Since $\alpha$ acts trivially on the cohomology we have $  \sign(S/\alpha) = \sign(S) = 0$. By \eqref{eq: equivariant signature alpha} we obtain $k_3=0$.
  This together with \eqref{eq: number of fixed points} yields
$e(S)=0$, a contradiction.

(iii) By Corollary~\ref{cor: G=4 k=2} the Albanese map $a_T\colon T\rightarrow B$ is a genus 2 fibration whose singular fibres are all of type (0). Let $\Sigma:=\proj (a_{T*}\omega_{T/B})\rightarrow B$ be the projectivized relative canonical bundle. By \cite{Hor77} the relative canonical map 
\[
 \xymatrix{
 T\ar[rr]^h\ar[dr]_{a_T} & & \Sigma\ar[dl]^{a_\Sigma}\\
  & B&
 }
\]
is a morphism of degree 2 and its branch curve $D\subset\Sigma$ has at most simple singularities. In fact, the curve $D$ is simple normal crossing by the following Lemma~\ref{lem: D}.

Since every ($-2$)-curve on $T$ is contracted to a singularity of $D$, there is a morphism $\bar\varphi\colon S/G_0\rightarrow \Sigma$ such that the morphism $h$ factors as $\bar\varphi\circ\lambda$ where $\lambda\colon T\rightarrow S/G_0$ is the minimal resolution of singularities. By Lemma~\ref{lem: D} the morphism $\bar\varphi\colon S/G_0\rightarrow \Sigma$ is a flat double cover of $\Sigma$ branched along $D$. The composition $\varphi=\bar\varphi\circ\pi\colon S\rightarrow \Sigma$ is then a finite morphism $\varphi\colon S\rightarrow \Sigma$ of degree 8, branched along $D$.

Over an analytic open subset $U$ of $\Sigma$  around the (nodal) singularities of $D$, the surface $S$ is a disjoint union of two $(\bz/2\bz)^2$-covers of $U$ (cf.~\cite[page~102]{BHPV04}). It is then direct to check that the singular fibres of $a_S\colon S\rightarrow B$ are of the form $2C$ with $C$ a smooth curve of genus 3. This together with the fact that $K_S^2=8\chi(\mo_S)$ guarantees that the surface $S$ is isogenous to a product of curves of unmixed type by \cite[Lemma~5]{Se95}.
\end{proof}

\begin{rmk}
 Using Corollary~\ref{cor: g5} one sees that the eight-to-one finite morphism $\varphi\colon S\rightarrow \Sigma$ in the proof of Theorem~\ref{thm: G=4 k=2} is indeed Galois.
\end{rmk}

\begin{lem}\label{lem: D}
Resume the notation in the proof of Theorem~\ref{thm: G=4 k=2}, and write $D=D_1 + \sum_{1\leq i\leq k}\Gamma_i$, where $D_1$ is horizontal with respect to the ruling $a_\Sigma\colon\Sigma\rightarrow B$ and the $\Gamma_i$'s are $k$ different fibres. Then the following holds.
\begin{enumerate}
 \item[(i)] The induced morphism $a_\Sigma\restr{D_1}\colon D_1\rightarrow B$ is \'etale. As a consequence, $D$ is a simple normal crossing curve.
 \item[(ii)] The morphism $\bar\varphi\colon S/G_0\rightarrow \Sigma$ is a \emph{flat} double covering.
\end{enumerate}
\end{lem}
\begin{proof}
(i) Let $e$ be the maximal integer such that there is a section of $a_\Sigma\colon\Sigma\rightarrow B$, say $\Delta$, with $\Delta^2=-e$. Then $e\geq -1$ (cf.~\cite[V.2]{Har77}). Numerically we can write 
\[
 D\equiv 6\Delta + d\Gamma\text{ and } K_\Sigma\equiv -2\Delta -e\Gamma,
\]
where $d$ is some integer. The arithmetic genus of $D$ depends only on its numerical class:
\begin{equation}\label{eq: paD1}
\begin{aligned}
p_a(D)&=1+\frac{D(D+K_\Sigma)}{2}\\
& = 1+\frac{(6\Delta+d\Gamma)\left(4\Delta+(d-e)\Gamma\right)}{2}\\
&=1+5d-15e.
\end{aligned}
\end{equation}
Note that
\begin{equation*}\label{eq: chi}
\begin{aligned}
\chi(\mo_S) =  \frac{1}{2}K_T^2 = \left(K_\Sigma + \frac{D}{2}\right)^2=\left(\Delta + \left(\frac{d}{2}-e\right)\Gamma\right)^2=d-3e.
\end{aligned}
\end{equation*}
where the first equality is by Corollary~\ref{cor: G=4 k=2}, so we have 
\begin{equation}\label{eq: paD3}
p_a(D)=5\chi(S,\mo_S)+1.
\end{equation}

Now we calculate $p_a(D)$ in another way. Let $\tilde\Sigma \rightarrow \Sigma$ be the minimal embedded resolution of the singularities of $D=D_1 + \sum_{1\leq i\leq k}\Gamma_i$ by blow-ups and $\tilde D\subset\tilde\Sigma$ the (smooth) strict transform of $D$.  Let $\tilde D_1$ and $\tilde \Gamma_i$ be the strict transform of $D_1$ and $\Gamma_i$ on $\tilde\Sigma$ respectively. Then $\tilde D=\tilde D_1 + \sum_{1\leq i\leq k}\tilde \Gamma_i$ and
\begin{equation}\label{eq: paD4}
\begin{aligned}
p_a(D) &= p_a(\tilde D) + \sum_{p\in D_\sing} \delta_p(D)\\
\end{aligned}
\end{equation}
where $D_\sing$ denotes the singular locus of the curve $D$ and, for $p\in D_\sing$, the number $\delta_p(D)$ is a positive integer determined locally by the type of the singularity $p\in D$.

There is a lower bound of $p_a(D)$:
\begin{equation}\label{eq: low bnd paD}
 \begin{aligned}
  p_a(D)  &\geq p_a(\tilde D) +6\chi(\mo_S)\hspace{1cm}(\text{by  \eqref{eq: paD4} and Lemma~\ref{lem: delp}})\\
  &= p_a(\tilde D_1) -k +6\chi(\mo_S)\\
&\geq 1 -k + 6\chi(\mo_S)\hspace{1cm} (\text{since }p_a(\tilde D_1)\geq 1)\\
 \end{aligned}
\end{equation}
Combining \eqref{eq: paD3} and \eqref{eq: low bnd paD} we see that
\begin{equation}\label{eq: k}
k\geq \chi(\mo_S) = d-3e.
\end{equation}

Numerically $D_1\equiv D-k\Gamma \equiv 6\Delta + (d-k)\Gamma$, so  
\begin{equation}\label{eq: D1}
 D_1\leq_\num 6\Delta + 3e\Gamma,
\end{equation}
which means that the divisor $6\Delta + 3e\Gamma-D_1$ is numerically equivalent to an effective $\bq$-divisor.

There are two cases depending on whether the rank two relative canonical bundle  $a_{T*}\omega_{T/B}$ is decomposable or not.

{\bf Case 1.} $a_{T*}\omega_{T/B}$ is decomposable. Then $e\geq 0$. If $e>0$ then there is no reduced curve in the linear system $|6\Delta + a_\Sigma^*\gb|$ for a divisor $\gb$ on $B$ of degree $3e$. Since $D_1$ is reduced, this case does not occur. If $e=0$ then $\Sigma=B\times\bp^1$ and $D_1\leq_\num 6\Delta$. Since $D_1$ is reduced, it is necessarily a union of disjoint 6 sections of $a_\Sigma\colon\Sigma\rightarrow B$. In particular, the induced morphism $a_\Sigma\restr{D_1}\colon D_1\rightarrow B$ is \'etale.

{\bf Case 2.} $a_{T*}\omega_{T/B}$ is indecomposable. Then the invariant $e$ of the corresponding ruled surface $\Sigma$ is $0$ or $-1$. As in \eqref{eq: paD1} we compute 
\begin{equation}\label{eq: paHor}
   p_a(6\Delta + n\Gamma) =1 -15e + 5n\,\text{ for } n\in\bz, \\
\end{equation}
which is less than 1 if $n<3e$. Since the arithmetic genus of any horizontal reduced curve in $\Sigma$ is at least $1$, no such curve is numerically equivalent to $6\Delta + n\Gamma$ for $n<3e$. As a horizontal reduced curve, $D_1$ must be numerically equivalent to $6\Delta + 3e\Gamma$ by \eqref{eq: D1}, so $p_a(D_1)=1$ by \eqref{eq: paHor}. This implies that $a_\Sigma\restr{D_1}\colon D_1\rightarrow B$ is \'etale by the Riemann--Hurwitz formula. 

Having shown that the morphism $a_\Sigma\restr{D_1}\colon D_1\rightarrow B$ is \'etale, the fibres $\Gamma_i$ ($1\leq i\leq k$) must intersect $D_1$ transversally. Therefore $D$ is a simple normal crossing curve.

(ii) By the proof of (i) we infer that $D_1\equiv 6\Delta + 3e\Gamma$ and the inequalities in \eqref{eq: low bnd paD} and \eqref{eq: k} are in fact equalities. Consequently, the equality case of Lemma~\ref{lem: delp} is achieved, so the singular locus of $S/G_0$ surjects onto $D_\sing$. From this the assertion of (ii) follows easily.
\end{proof}

\begin{lem}\label{lem: delp}
With the same notation as in the proofs of Theorem~\ref{thm: G=4 k=2}, (iii) and Lemma~\ref{lem: D} we have 
\begin{equation}\label{eq: delp}
\sum_{p\in D_\sing} \delta_p(D) \geq 6\chi(\mo_S).
\end{equation}
with equality only if the singular locus of $S/G_0$ surjects onto $D_\sing$.
\end{lem}
\begin{proof}
We know by Theorem~\ref{thm: G=4 k=2}, (ii) that $G_0\cong(\bz/2\bz)^2$. Let $\sigma_1,\sigma_2,\sigma_3$ be the three involutions of $G_0$. Since the sets $S^{\sigma_i}$ ($1\leq i\leq 3$) are finite by Lemma~\ref{lem: iso}, they  are pairwise disjoint by Lemma~\ref{lem: Z22 acts}. 

By the topological Lefschetz fixed point formula \eqref{eq: lefschetz} there are $e(S)=4\chi(\mo_S)$ points fixed by each involution $\sigma_i\in G_0$, $ 1\leq i\leq 3$.  The image of $S^{\sigma_i}$ under the quotient map $\pi\colon S\rightarrow S/G_0$ consists of $2\chi$ singularities of type $A_1$, where $\chi:=\chi(\mo_S)$. Resolving those singularities we obtain totally $6\chi$ disjoint ($-2$)-curves $E_l^{(i)}$ on $T$
with $1\leq i\leq 3$ and $1\leq l\leq 2\chi$, where $E_l^{(i)}$ ($1\leq l\leq 2\chi$) lie over the points in $\pi(S^{\sigma_i})$.

Let $\tilde S\rightarrow S$ be the simultaneous blow-up of $S$ at the points of $\bigcup_{1\leq i\leq 3}S^{\sigma_i}$. Then the induced morphism $\tilde S\rightarrow T$ is a bidouble cover, branched exactly along the $ E_l^{(i)}$'s with $1\leq i\leq 3 $ and $1\leq l\leq 2\chi$. Moreover, the stabilizer over the curves $E_l^{(i)}$ is $\sigma_i$. It follows from the theory of bidouble covers \cite{Cat84} that, for $1\leq i< j\leq 3$ the divisor $\sum_{1\leq l\leq 2\chi} (E_l^{(i)}+E_l^{(j)})$ is even, meaning that it is linearly equivalent to $2L$ for some divisor $L$. An even divisor has necessarily an even intersection number with each curve.

Note that each ($-2$)-curve $E_l^{(i)}$ is contracted to some singularity of $D$ under the morphism $h\colon T \rightarrow \Sigma$. So it suffices to show that, over each $p\in D_\sing$, we have
\begin{equation}\label{eq: delp}
\delta_p(D)\geq \#\,\left\{E_l^{(i)}, 1\leq i\leq 3, 1\leq l\leq 2k \,\middle|\, E_l^{(i)}\text{ is contracted to }p\right\}.
\end{equation}

We determine $\delta_p(D)$ according to the type of the singularity $p\in D$ as in the following table:
{\renewcommand{\arraystretch}{1.2} 
\begin{center}
 \begin{tabular}{|c|c|c|c|c|c|c|}\hline
type of $p\in D$ &  $A_n, n\geq 1$ & $D_n, n\geq 4$ & $E_6$ & $E_7$ & $E_8$\\ \hline
$\delta_p(D)$  & $\lfloor \frac{n+1}{2}\rfloor$ & $1+\lfloor \frac{n}{2}\rfloor$ & 3 & $4$ & $4$\\ \hline
\end{tabular}
\end{center}
}

If $p\in D$ is of type $A_n$ with $n$ even or of type $E_n$ with $6\leq n\leq 8$ then there is no non-empty collection of disjoint ($-2$)-curves on $T$, which lie over $p$ and whose sum has an even intersection number with each component over $p$. So in this case the right hand side of \eqref{eq: delp} is $0<\delta_p(D)$.

If $p\in D$ is of type $A_n$  with $n$ odd then there is exactly one non-empty collection of disjoint ($-2$)-curves on $T$, which lie over $p$ and whose sum has an even intersection number with each component over $p$. In terms of the following dual graph of the ($-2$)-curves on $T$ lying over $p$, the non-empty even collection consists of the curves corresponding to the $\circ$'s and has cardinality $\lfloor \frac{n+1}{2}\rfloor=\delta_p(D)$:
\[
 \xymatrix{
 {\circ}\ar@{-}[r]& {\bullet}\ar@{-}[r]&{\circ}\ar@{-}[r]&\ar@{.}[r]&\ar@{-}[r]&{\bullet}\ar@{-}[r] & {\circ}
 }
\]

Similarly, if $p\in D$ is of type $D_n$ with $n\geq 5$ then there is exactly one non-empty collection of disjoint ($-2$)-curves on $T$, which lie over $p$ and whose sum has an even intersection number with each component over $p$. In terms of the following dual graph of the ($-2$)-curves on $T$ lying over $p$, the non-empty even collection consists of the curves corresponding to the $\circ$'s and has cardinality $2<\delta_p(D)$:
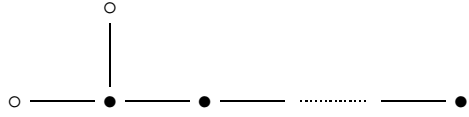
\[
 \xymatrix{
& {\circ}\ar@{-}[d]&&&\\
  {\circ}\ar@{-}[r]& {\bullet}\ar@{-}[r]&{\bullet}\ar@{-}[r]&\ar@{.}[r]&\ar@{-}[r]&{\bullet}
 }
\]

If $p\in D$ is of type $D_4$ then there are three non-empty collections of disjoint ($-2$)-curves on $T$, which lie over $p$ and whose sum has an even intersection number with each component over $p$. In terms of the following dual graph of the ($-2$)-curves on $T$ lying over $p$, each of these collections consists of two of the three curves corresponding to the $\circ$'s:
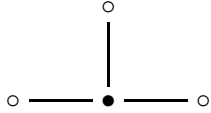
\[
 \xymatrix{
& {\circ}\ar@{-}[d]&&&\\
  {\circ}\ar@{-}[r]& {\bullet}\ar@{-}[r]&{\circ}
 }
\]
so in this case
\[
\#\left\{E_l^{(i)}\,\middle|\, E_l^{(i)}\text{ is contracted to }p\right\}\leq\#\left\{\text{$\circ$'s in the dual graph}\right\}=3=\delta_p(D).
 \]
\end{proof}

\section{Surfaces with quotient not of general type}\label{sec: k=1} Let $S$ be a minimal surface of general type with $q(S)=1$. In this section we assume that the Kodaira dimension of $S/G_0$ is 1, where we write $G_0$ for $\au(S)$. More notation is resumed from the diagram \eqref{diag: resol}. We will let $F$ denote a smooth fibre of the Albanese fibration $a_S\colon S\rightarrow B$, where $B$ is the identified Albanese varieties $\alb(S)$ and $\alb(T)$, see \eqref{diag: alb}.

\subsection{Bounding $|\au(S)|$, part II}\label{sec: bound}
We are going to describe the canonical systems of the surfaces in the diagram \eqref{diag: resol}. For a surface $X$ we use $\omega_X$ and $\mo_X(K_X)$ for the canonical sheaf interchangeably.

Since $\kappa(T)=1$, the Albanese map $a_T\colon T\rightarrow B$ is an elliptic fibration by Lemma~\ref{lem: g alb}. The canonical bundle formula for
relative minimal elliptic fibrations gives
\begin{equation}\label{eq: canonical bundle formula}
\mo_T(K_T) = a_T^* \mathfrak{b} \otimes \mathcal O_T\left(\sum_i(m_i-1)F_i\right)
\end{equation}
where $\mathfrak{b}=(R^1{a_T}_*\mathcal O_T)^\vee$ is an invertible sheaf of degree $\chi(\mo_T)=\chi(\mo_S)$ on $B$  and the $m_iF_i$'s are the multiple fibres of $a_T$.

Since  $\tilde T$ is obtained from $T$ by successively blowing up
smooth points,  there is some effective divisor $A$, supported on the whole exceptional locus of $\eta\colon
\tilde T\rightarrow T$, such that
\begin{equation}\label{eq: adjunction T}
\mo_{\tilde T}(K_{\tilde T})=\mo_{\tilde T}\left(\eta^*K_T+A\right).
\end{equation}
By \eqref{eq: canonical bundle formula} and \eqref{eq: adjunction T} we have
\begin{equation}\label{eq: canonical bundle of tilde T}
 \mo_{\tilde T} (K_{\tilde T})=
 a_{\tilde T}^* \mathfrak{b} \otimes \mathcal O_{\tilde T}
 \left(A+\sum_i(m_i-1)\eta^*F_i\right).
\end{equation}

Let $\tilde R = K_{\tilde S} - \tilde \pi^* K_{\tilde T}$ be
the ramification divisor of $\tilde\pi$. By Lemma~\ref{lem: hol inv quot} there are equations between complete linear systems
\begin{equation}\label{eq: canonical system tilde S}
|K_{\tilde S}| = \tilde\pi^*|K_{\tilde T}| +\tilde R= a_{\tilde S}^*|\mathfrak{b}|
 + \tilde R + \tilde\pi^*A + \sum_i(m_i-1)\tilde\pi^*\eta^*F_i,
\end{equation}
where $a_{\tilde S}\colon \tilde S\rightarrow B$ is the Albanese map of
$\tilde S$.

Since $\rho\colon \tilde S\rightarrow S$ is a composition of
blow-ups at smooth points, it is well known that $|K_S|= \rho_*|K_{\tilde S}|$. So
\eqref{eq: canonical system tilde S} gives
\begin{equation}\label{eq: can sys S}
 |K_S| = \rho_*|K_{\tilde S}| = a_S^*|\mathfrak{b}| + \rho_*\tilde R +
 \rho_*\tilde\pi^*A + \sum_i(m_i-1)\rho_*\tilde\pi^*\eta^*F_i.
\end{equation}
Note that $\rho_*\tilde R$ is just the ramification divisor
 $R=K_S-\pi^*K_{S/G}$ of the quotient map $\pi\colon S\rightarrow S/G$. Every irreducible component of $R$ is fixed (pointwise)
 by some nontrivial element of $G$, hence is smooth.

\begin{rmk}
From \eqref{eq: can sys S} we see that the
Albanese map  $a_S\colon S\rightarrow B$ is induced by the canonical system of $S$ if $p_g(S)>1$ and by the paracanonical system of $S$ if $p_g(S)=1$ (cf.~\cite{CC91}).
\end{rmk}

\begin{nota}\label{nota: HV}
 Set $M=a_S^*\mathfrak{b}$ and $Z=\rho_*\tilde R + \rho_*\tilde\pi^*A
+ \sum_i(m_i-1)\rho_*\tilde\pi^*\eta^*F_i$. 

Then \eqref{eq: can sys S} reads $|K_S|=|M|+Z$. The divisor $M$ is
algebraically equivalent to $\chi(\mo_S)F$, and it moves if
and only if $\chi(\mo_S)>1$. On the other hand, $Z$ always
belongs to the fixed part of $|K_S|$.
 
 We write $Z=H+V$, and $H=n_1\Gamma_1+\dots + n_t
\Gamma_t$ with $n_1\geq\dots\geq n_t$, where $H$ (resp.
$V$) is the horizontal part (resp. the vertical part) of $Z$ with
respect to the Albanese fibration $a_S\colon S\rightarrow B$, and the $\Gamma_i$'s are  the irreducible components of $H$, with  $n_i$ being the  multiplicity of $\Gamma_i$ in $H$. Observe that
\begin{equation}\label{eq: HF}
 2g(F) -2 = K_SF = HF = \sum_{1\leq i\leq t} n_i\Gamma_i F.
\end{equation}
\end{nota}

Obviously, the part $\rho_*\tilde\pi^*A +
\sum_i(m_i-1)\rho_*\tilde\pi^*\eta^*F_i$ of $Z$ is contained in $V$,  so we have $H<\rho_*\tilde R=R$. In particular, its
irreducible components $\Gamma_i$ ($1\leq i\leq t$) are all smooth.
Moreover, $n_i+1$ is the ramification index of the quotient map $S\rightarrow S/G_0$ (equivalently, the order of the stabilizer of the $G_0$-action on $S$) at a general
point of $\Gamma_i$. Since $a_S\rest{\Gamma_i}\colon\Gamma_i\rightarrow B$ is dominant we
have $g(\Gamma_i)\geq g(B)=1$  for all $i$.

\begin{lem}\label{lem: MH}
 $MH=(2g(F)-2)\chi(\mo_S)$.
\end{lem}
\begin{proof}
We compute 
\[
 MH = \chi(\mo_S) FH = \chi(\mo_S) FK_S = (2g(F)-2)\chi(\mo_S).
\]
\end{proof}

\begin{lem}\label{lem: KH}
There are the following bounds on the $K_SH$, the canonical degree of $H$:
\[
 \frac{2g(F)-2}{n_1+1}\chi(\mo_S)+
\sum_{i=1}^t \frac{2n_i^2}{ n_1+1}(g(\Gamma_i)-1) \leq K_S H \leq (11-2g(F)) \chi(\mo_S)
\]
where the second inequality is strict if $g(\Gamma_i)=1$ for some $i$.
\end{lem}
\begin{proof} 
Since $n_1\geq n_i$ for all $1\leq i\leq t$ by assumption, we have
\[
 (n_1K_S+H+V)\Gamma_i \geq (n_1 K_S + n_i \Gamma_i)\Gamma_i \geq
n_i(2g(\Gamma_i)-2).
\]
So
\begin{align*}
(n_1+1)K_S H-MH=(n_1K_S + H +V)H \geq  \sum_{i=1}^t n_i^2(2g(\Gamma_i)-2)
\end{align*}
and the first inequality follows by plugging in the formula of Lemma~\ref{lem: MH}.

We compute further
\begin{multline}\label{eq: low bound K^2}
  K_S^2=K_S(M+H+V)\geq K_SM + K_SH\\ \geq MH + K_SH =(2g(F)-2)\chi(\mo_S)+K_SH.
\end{multline}
Combining this with the Bogomolov--Miyaoka--Yau inequality
$K_S^2\leq 9\chi(\mo_S)$ we obtain $K_S H\leq
(11-2g(F))\chi(\mo_S)$.

Now suppose $g(\Gamma_i) =1$ for some $i$. Containing an elliptic curve, the surface $S$ cannot be a ball quotient (otherwise the elliptic curve will lift to the ball, which is absurd). Hence $K_S^2<9\chi(\mo_S)$ by Yau's result \cite{Y77}. By \eqref{eq: low bound K^2} we infer that the second inequality is strict in this case.
\end{proof}

The following bound on the genus of the fibration $a_S\colon S\rightarrow B$ is in the same spirit of \cite[Sec.~2]{Be79}.
\begin{cor}\label{cor: KH}
 We have $\frac{2g(F)-2}{n_1+1}<11-2g(F)$. In particular, $g(F)\leq 5$.
\end{cor}

By Lemma~\ref{lem: alb quot} we can analyze the action of $G_0$ on $S$ by restricting to a general fibre $F$ of the Albanese fibration $a_S\colon S\rightarrow B$, where the Riemann--Hurwitz formula applies.

Assume that the quotient map $\pi\restr{F}\colon F\rightarrow F/G_0$ is branched at $k$ points, over which the ramification indices are $r_1,\dots,r_k$ respectively.  The following variant of the Riemann--Hurwitz formula will be used repeatedly in the proof of Theorem~\ref{thm: bound k=1}:
\begin{equation}\label{eq: RH}
\frac{2g(F)-2}{|G_0|}=\sum_{i=1}^k\left(1-\frac{1}{r_i}\right),
\end{equation}
the right hand side of which is at least $1$ if $G_0$ is abelian.

\begin{thm}\label{thm: bound k=1}
Let $S$ be a minimal surface of general type with $q(S)=1$ such that $\kappa(S/G_0)=1$, where $G_0:=\au(S)$. Then $|G_0|\leq 4$, and if the equality holds then the Albanese fibration of $S$ has genus $3$.
\end{thm}
\begin{proof}
By Corollary~\ref{cor: KH} we have $g(F)\leq 5$. We distinguish four cases according to the value of $g(F)$.

{\bf Case 1.} $g(F)=5$. We will show that this case does not occur.

 By Corollary~\ref{cor: KH} we have $n_1\geq 8$. By \eqref{eq: HF}
it must hold $n_1=8$, $\Gamma_1F=1$ and $H=\Gamma_1$. So $\Gamma_1$ is a section of $a_S$. By the adjunction formula we have $K_S\Gamma_1+\Gamma_1^2=0$. Since $K_S$ is big and nef the Hodge index theorem implies that $\Gamma_1^2<0$. Hence $\Gamma_1$ is $G_0$-fixed by Lemma~\ref{lem: neg}. Locally around a general point of $\Gamma_1$ the group action takes the form $\sigma(x,y)\mapsto (x, \chi(\sigma)y)$ for $\sigma\in G_0$, where $y=0$ is defining equation of $\Gamma_1$  and $\chi$ is a character embedding $G_0$ into $\bc^*$, so $G_0$ must be cyclic and its order is $n_1+1=9$ by the discussion before Lemma~\ref{lem: MH}. This results in a contradiction to \eqref{eq: RH}.

{\bf Case 2.} $g(F)=4$.  We will show that   $|G_0|\leq3$ in this case.

By  Corollary~\ref{cor: KH} we have $n_1\geq2$, so the ramification index at $\Gamma_1\cap F$ is $r_1=n_1+1\geq3$.
 If $|G_0|\geq 4$ then there are two
possibilities by \eqref{eq: RH}:
\begin{enumerate}
 \item[(a)] $|G_0|=8$, $r_1=4$;
 \item[(b)] $|G_0|=4$, $r_1=r_2=4$.
\end{enumerate}

In the case (a),  let $\gamma$ be the generator of the monodromy around
the branch point $q$ of  $\pi\restr{F}\colon F\rightarrow F/G_0$. Then $\gamma$ has order $r_1=4$. Since $g(F/G_0)=1$,   the fundamental group $\pi_1(F/G_0 \setminus \{q\})$ has a representation $\langle a, b ,c\,|\, aba^{-1}b^{-1}c=1\rangle$ with $a,b$ being generators of $\pi_1(F/G_0)$ and $c$ a small loop around $q$, so that the image of $c$ under the quotient map $\pi_1(F/G_0 \setminus \{q\})\rightarrow G_0$ is $\gamma$.  It follows that $\gamma$ is  a commutator of $G_0$. On the other hand, one sees easily that any commutator of a
group of order $8$ has order at most $2$. So this case is excluded.

In the case (b), since $G_0$ contains elements of order $4$, it must be isomorphic to $\bz/4\bz$. Moreover, $n_i=3$ for all $i$. Then there are two possibilities for $H$ by \eqref{eq: HF}:
\begin{enumerate}
\item[(i)] $H=3(\Gamma_1+\Gamma_2)$ with $\Gamma_1F=\Gamma_2F=1$;
\item[(ii)]  $H=3\Gamma_1$ with $\Gamma_1F=2$.
\end{enumerate}

Let $H_\red$ be  the reduced part of $H$.
 Since $H_\red$ is fixed by $G_0$,
it is smooth and hence in the case (i) the two curves $\Gamma_1$ and
$\Gamma_2$ do not intersect. We claim that
\begin{align}\label{eq: h}H_\red^2<0.\end{align}
Indeed, if $H$ is in the case (i) then, noting that the self-intersection of a
section of $a_S$ is negative (cf.~the proof of {\bf Case 1}), we have 
\[
 H_\red^2 = \Gamma_1^2 + \Gamma_2^2<0.
\]
Now assume  $H$ is in the case (ii). If $\Gamma_1^2\geq0$, then
\[
 K_SH=3(M+3\Gamma_1+V)\Gamma_1\geq 3M\Gamma_1=6\chi (\mo_S),
\]
a contradiction to the second inequality of Lemma~\ref{lem: KH}. This finishes the proof of (\ref{eq: h}).

Let $\sigma\in G_0$ be the involution.  Since each
$\sigma$-fixed curve other than $H_\red$ is contained in fibers of $a_S$ we have, by Lemma~\ref{lem: inv} and (\ref{eq: h}),
\begin{align}\label{eq: g4}K_S^2\leq8\chi (\mo_S)+H_\red^2
< 8\chi (\mo_S).\end{align}

On the other hand, let $X$ be the minimal model of $S/\sigma$
and $a_X:X\to B$ the Albanese fibration. For a general $b\in B$ the fibre $a_X^*b$ is the quotient of $F_b:=a_S^*b$ by $\sigma$.  Since $F_b\to F_b/\sigma$ is ramified exactly at two points, the genus of $a_X$ is $2$ by the Riemann--Hurwitz formula. Thus ${a_X}_*\omega_X$ is a rank two vector bundle of degree $\chi(\mo_S)$ over $B$ and the associated projective bundle $P:=\proj({a_X}_* \omega_X)$ is a ruled surface over $B$. Denote by $e$ the largest number such that there is a section $\Delta$ of $a_P\colon P\rightarrow B$ with $\Delta^2=-e$.
 We have (cf.~\cite[page~7]{X85b}):
 \begin{align}\label{rule}
e=\max \{2\deg\mathcal L - \deg {a_X}_*\omega_X\mid \mathcal
L\subset {a_X}_*\omega_X\text{ is a sub-line-bundle}\}.\end{align}

Since  $a_X\colon X\rightarrow B$ is the composition of the induced rational map $X\dashrightarrow T$ with $a_T\colon T\rightarrow B$, the rank two vector bundle ${a_X}_*\omega_X$ contains the line bundle ${a_T}_*\omega_T$.   Note that $\deg{a_T}_*\omega_T= \chi(\mo_T) =\chi (\mo_S)$, so we have by (\ref{rule}) $$e\geq 2\deg  {a_T}_*\omega_T- \deg {a_X}_*\omega_X=\chi (\mo_S).$$
Therefore, by \cite[Thm~2.2, (ii)]{X85b}, we have
\begin{equation}\label{eq: x}
 K_X^2 \geq \chi(\mo_X) + 3 e \geq 4\chi(\mo_X).
 \end{equation}

Since $a_X$ has genus $ 2$, it follows that $X$ is of general type
by Lemma~\ref{lem: g alb}. By Proposition~\ref{prop: X94} and (\ref{eq: x}), we obtain
$$K_S^2\geq 2K_X^2\geq8\chi (\mo_X)=8\chi (\mo_S),$$
which is a contradiction to \eqref{eq: g4}. Thus we exclude the
case (b).

{\bf Case 3.} $g=3$.  We assume $|G_0|>4$.  Then $G_0\cong D_6$, $D_8$, or $Q_8$ by \eqref{eq: RH}. If $G_0\cong D_6$ or $D_8$, by the proof of \cite[Claim 3.8]{Cai04}, we have
\[
 K_S^2=\frac{8}{3} \chi(\mo_S)-\frac{32}{3}(g(B)-1)= \frac{8}{3}\chi(\mo_S),
\]
a contradiction to (\ref{eq: low bound K^2}); if $G_0\cong Q_8$, by the proof of \cite[Claim 3.7]{Cai04}, we have 
\[
 K_S^2 =  3 \chi(\mo_S)+10(g(B)-1)=3 \chi(\mo_S),
\]
again a contradiction to (\ref{eq: low bound K^2}).\footnote{Note that in the proofs of \cite[Claims~3.7 and 3.8]{Cai04} one does not need any condition on $\chi(\mo_S)$ as required by the main theorem of \cite{Cai04}.} So in the case $g=3$ we
have $|G_0|\leq4$.

{\bf Case 4.} $g=2$.  In this case we have $|G_0|\leq2$ by \eqref{eq: RH}.

This finishes the proof of Theorem~\ref{thm: bound k=1}.

\end{proof}

\subsection{Surfaces with $|\au(S)|=4$, part II}\label{sec: isog}
In this subsection we will prove the following theorem.
\begin{thm}\label{thm: isog k=1}
Let $S$ be a minimal surface of general type with $q(S)=1$ such that $\kappa(S/G_0)=1$ and $|G_0|=4$, where we denote by $G_0$ the group $\au(S)$. Then $S$ is isogenous to a product of curves of unmixed type.
\end{thm}

We need some preparation for the proof of Theorem~\ref{thm: isog k=1}, which will be given in the end of this subsection.

By Theorem~\ref{thm: bound k=1} the Albanese map
$a_S: S\to B$ has genus $3$. By \eqref{eq: RH} there are exactly 2 branch points of the quotient map $\pi\restr{F}\colon F\rightarrow F/G_0$ and the ramification indices thereover are both 2. Therefore the horizontal part $H$ of the divisors from $|K_S|$ is a reduced curve with $HF=2g(F)-2=4$.

Let  $\gamma_1$, $\gamma_2$ be the stabilizers over the two branch points of $\pi\restr{F}\colon F\rightarrow F/G_0$. Looking at the monodromy we see that $\gamma_1\gamma_2=\id_F$ and hence $\gamma_1=\gamma_2$. Denote by $\sigma$ the $\gamma_i$,
$i=1,2$. Then $H$ is $\sigma$-fixed. It is also easy to see that
\begin{equation}\label{eq: genus F/sig}
 g(F/\sigma) = g(F/G_0)=1.
\end{equation}

\begin{lem}\label{lem: V=0}
$\mathrm{(i)}$ $K_S^2 = 8 \chi(\mo_S) + H^2$; $\mathrm{(ii)}$ $V=0$.
\end{lem}
\begin{proof}
We compute
\begin{equation}\label{eq: g30}
\begin{split}
  K_S^2 &= K_S(M+H+V) \\
  &= MH+K_SH +K_SV\\
  &=MH+(M+H+V)H +K_SV\\
   &= 8\chi(\mo_S) + H^2+(H+K_S)V\hspace{1cm}(\text{by Lemma }\ref{lem: MH})\\
&\geq 8\chi(\mo_S) + H^2.
\end{split}
\end{equation}
Since each $\sigma$-fixed curve other than $H$ is contained in fibers of $a_S$, we have by Lemma~\ref{lem: inv} 
\begin{align}\label{eq: g31} K_S^2\leq8\chi (\mo_S)+H^2.\end{align}
Combining  \eqref{eq: g30} with \eqref{eq: g31} we obtain
\begin{align}
 \label{eq: g32}&K_S^2=8\chi (\mo_S) +H^2,\\
 \label{eq: g32'}&(H+K_S)V=0.
\end{align}
 From (\ref{eq: g32'}) follows $HV=0$ and hence  $(M+H)V=0$. This implies $V=0$ since effective canonical divisors are $2$-connected  (\cite[VII, Prop.~6.2]{BHPV04}).
\end{proof}
\begin{cor}\label{cor: V=0}
The curve $H$ is the only curve that is fixed by a nontrivial automorphism in $G_0$.
\end{cor}
\begin{proof}
If there were another curve $C$ that is fixed by a nontrivial automorphism in $G_0$ then $C\leq V$, which is a contradiction to Lemma~\ref{lem: V=0}, (ii).
\end{proof}

 \begin{cor}\label{cor: X  T min}
The minimal resolution $\tilde T$ of $S/G_0$ in \eqref{diag: resol} is minimal, that is, it does not contain any $(-1)$-curves. As a consequence $\tilde T$ does not contain any ($-4$)-curves.
 \end{cor}
\begin{proof}
If $\tilde T$ is not minimal then there is a $(-1)$-curve
$E$ on it, which is necessarily not contracted by $\lambda$. The curve $\lambda(E)\subset S/G_0$ is then pulled back to be a fixed part of $|K_S|$, which is contained in $V$. This is a contradiction to Lemma~\ref{lem: V=0}.

For the second statement note that a relatively minimal elliptic fiberation cannot contain any ($-4$)-curves in the fibres 
 for example by Kodaira's classification of singular elliptic fibres (\cite[V.7]{BHPV04}).
\end{proof}

Now let $\mu\colon H\to B$ be the restriction of $a_S$
to $H$. Then $\mu$ is a finite map  of degree 4. We remark that
the degree and the ramification divisor of $\mu$ make sense even
when $H$ is not connected.

Let $R=K_H -\mu^* K_B$ be the ramification divisor of the four-to-one morphism $\mu\colon H\rightarrow B$. We may write $R=\sum_{b\in B} R_b$ where $R_b$ is the part of $R$ over $b\in B$. By the Riemann--Hurwitz formula and the adjunction formula, 
 \begin{align}\label{r0} \deg R=H^2+HK_S.\end{align}
\begin{lem}\label{lem: F_b}
For each point $p\in S$, if  $p\leq R$ as effective divisors, then $F_b:=a_S^*b$ is singular
at $p$, where $b=a_S(p)$.
\end{lem}
\begin{proof}
There are local coordinates $(x,y)$ such that
\[
 \sigma(x,y) = (x,-y).
\]
Moreover,  we may assume $H$ is locally defined by $y=0$ and $F_b$ by
$$c_1 x + c_2x^2+ c_3y^2 + c_4 x y^2 + \text{higher order terms} =0$$
where $c_i\in \bc$ are constants. The assumption implies that the intersection number of $H$ and $F_b$ at $p$ is at least 2, so we have $c_1 =0$ and the lemma follows.
\end{proof}

\begin{lem}\label{lem: bou sin fib}
For each branch point $b\in B$ of $\mu$, let $\epsilon(F_b)=e(F_b)+4$ be the topological defect of the fibre $F_b:=a_S^*b$, see Appendix~\ref{sec: top def}. Then we have $\epsilon(F_b)\geq \deg R_b$, and equality holds only if $R_b=p+q$ with $p\neq q$. 
\end{lem}
\begin{proof}
We distinguish the two cases $G_0\cong(\bz/2\bz)^2$ and $G_0\cong\bz/4\bz$.

{\bf Case 1. $G_0\cong(\bz/2\bz)^2$.} We show that the morphism $\mu\colon H\rightarrow B$ is a bidouble cover. By Lemma~\ref{lem: g3} the fibration $a_S\colon S\rightarrow B$ is hyperelliptic. Let $\tau$ be the hyperelliptic involution, which is necessarily not in $G_0$. Let $G$ be the subgroup of  $\aut(S)$ generated by $\tau$ and $G_0$. Since the hyperelliptic involution of a curve of genus at least $2$ commutes with all of its automorphisms, we have $G\cong(\bz/2\bz)^3$.

Denote by $G_H$ the the image of $G$ in $\aut(H)$, which is isomorphic to $(\bz/2\bz)^2$. Since $\mu\colon
H\rightarrow B$ has degree 4 and factors through the quotient map $H\rightarrow H/G_H$ which also has degree 4, the two maps coincide. In particular, $\mu$ is Galois with Galois group
$G_H\cong(\bz/2\bz)^2$. 

It follows that, for each branch point $b\in B$
of $\mu$, the inverse image $\mu^{-1}(b)$ consists of two points, say $p$ and $q$, and we have $R_b=p+q$. By Lemma~\ref{lem: F_b}, the fibre $F_b$ is singular at both $p$ and $q$. On the other hand, if $\epsilon(F_b)\leq1$ then  $F_b$ has at most one
singular point by Lemma~\ref{lem: e_f=1} for $g=3$. So we have $\epsilon(F_b)\geq2=\deg R_b$.

{\bf Case 2. $G_0\cong\bz/4\bz$.} Since the restriction $\alpha_H:=\alpha\rest{H}$ is an involution of $H$ and $\mu: H\to B$ factors through the quotient map $H\to H/\alpha_H$, $R_b$ is of the form either $p$, $p+q$ ($p\not= q$) or $3p$. Now the lemma follows from the statements below which we will prove case by case: 
\begin{enumerate}
 \item[(i)] if $R_b=p$ then $\epsilon(F_b)>1$;
\item[(ii)] if $R_b=p+q$ then $\epsilon(F_b)\geq 2$;
 \item[(iii)] if $R_b=3p$ then $\epsilon(F_b)>3$.
\end{enumerate}

(i) If $R_b=p$ then the point $p$ is $\alpha$-fixed. Since the curve $H$ is $\sigma$-fixed, the action of $\alpha$ at   $p\in S$ is of weight $\frac{1}{ 4}(1, 2)$.
By Lemmata~\ref{lem: F_b} and \ref{lem: fixed}, $p$ is neither a smooth  point nor an ordinary  node of $F_b$. So $F_b$ cannot be as in Lemma~\ref{lem: e_f=1} and we have $\epsilon(F_b)\geq2$.

(ii) If $R_b=p+q$ then $p$ and $q$ are two singular points of $F_b$ by Lemma~\ref{lem: F_b}. Therefore $F_b$ cannot be as in Lemma~\ref{lem: e_f=1} for $g=3$ and we have $\epsilon(F_b)\geq2=\deg R_b$.

(iii) Since $R_b=3p$, it follows from the fact $HF_b=4$ that $H\cap F_b =\{p\}$ and the intersection number of $H$ and $F_b$ at $p$
is $4$.  

We look at the action of $\alpha$ around $p\in S$ which is necessarily of type $\frac{1}{4}(1,2)$. There are suitable local coordinates $(x,y)$ of $S$ around $p$ such that $H\subset S$ is defined by $x=0$ and $\alpha$ acts as $\alpha(x,y)=(\pm\sqrt{-1} x, -y)$.

Let $t$ be a local coordinate of $B$ around the point $b$. Then, as a holomorphic function around $p$, the pull-back $a_S^*t$ is invariant under the action of $\alpha$ and takes the following form in local coordinates:
\begin{align}\label{f^t}
 a_S^*t=c_1y^2 + c_2x^2y + c_3 y^4 + c_4x^4  + \text{ higher order terms},
\end{align}
where $c_i\in\bc$ are constants. Since the intersection number of $H$ and $F_b$ at $p$ is $4$, we have $c_1=0$ and hence the multiplicity $\mu_p(F_b)\geq 3$.

 On the other hand, let $F_\red$ be the reduced part of $F_b$. Then we have by Lemma~\ref{lem: delta Fb}
$$
\epsilon(F_b)=\epsilon(F_\red)+2p_a(F_b)-2p_a(F_\red)=\epsilon(F_\red)+6-2p_a(F_\red).$$
If $\epsilon(F_b)\leq 3$ then we have either $p_a(F_\red)=2$
and $\epsilon(F_\red)\leq1$ or $p_a(F_\red)=3$ and $\epsilon(F_\red)\le 3$. 

In the first case $F_\red$ cannot be smooth, otherwise $F_b=2F_\red$ has multiplicity two at any points, contracting the fact that $\mu_p(F_b)\geq 3$. It follows that $\epsilon(F_\red)=1$ and hence $F_\red$ has a unique node $p$ as singularity by Lemma~\ref{lem: delta reduced}. In particular, $F_b$ has at most two components. We claim that $F_b=mF_\red$ for some positive integer $m$. If $F_b$ is irreducible this is clear. Otherwise $F_b$ has two components, say $C_1$ and $C_2$. Since $F_bC_1=F_bC_2=0$ and $C_1C_2=1$, the multiplicities of $C_1$ and $C_2$ are necessarily the same.  Given $p_a(F_b)=3$ and $p_a(F_\red)=2$, we see that $F_b=2F_\red$. In terms of local coordinates $(x,y)$ around $p$ above:
\[
a_S^*t=\left((ax+by)(cx+dy)+\textrm{terms of higher  order}\right)^2 
\]
 for some $a, b, c, d\in\bc$ with $ad-bc\not=0$. This is a contradiction to (\ref{f^t}). 
 
In the second case $F_b=F_\red$ is reduced. Since
$\mu_p(F_b)\geq3$, we have by Lemma~\ref{lem: delta  reduced} that $\epsilon(F_b)\geq  4 >3$.

\end{proof}

\begin{prop}\label{prop: 8X}
 $K_S^2=8\chi(\mo_S)$, or equivalently $e(S)=4\chi(\mo_S)$.
\end{prop}
\begin{proof}
First assume $G_0\cong(\bz/2\bz)^2$. Note that the curve $H$ is $\sigma$-fixed and the other involution $\sigma_1$ and $\sigma_2$ of $G_0$ do not fixed any curves (Corollary~\ref{cor: V=0}). By Lemma~\ref{lem: inv}, applied to the involution $\sigma_1$ or $\sigma_2$, we have the equality $K_S^2=8\chi (\mo_S)$ .

Now assume $G_0=\langle\alpha\rangle\cong \bz/4\bz$. Note that $H$ is  $\sigma$-fixed but not $\alpha$-fixed. Applying the equivariant signature formula to $\alpha$ (\cite[1.6]{Cai09}), we have
\begin{align}\label{eq: equiv sign 1}
4\op{Sign}(S/\alpha)=\op{Sign}(S)+H^2+\sum_{p\in
S}\op{def}_p(S,\alpha),
\end{align}
where 
\[
\op{def}_p(S,\alpha) =
 \begin{cases}
 2 & \text{ if $\alpha$ has weight $\frac{1}{4}(1,3)$ at } p\in S,\\
 -2 & \text{ if $\alpha$ has weight $\frac{1}{4}(1,1)$ at } p\in S,\\
0 & \text{ otherwise.}
 \end{cases}
\]
Since $\sigma$  acts trivially on $H^2(S, \br)$, we infer that $\op{Sign}(S/\sigma)=\op{Sign}(S)$ and hence by \eqref{eq: equiv sign 1}
\begin{align}\label{eq: equiv sign 2}
3\left(K_S^2-8\chi (\mo_S)\right)=3\op{Sign}(S)=H^2+2(k_3-k_1),
\end{align} 
where $k_a$ ($a=1, 3$) is the number of isolated $\alpha$-fixed points of weight $\frac{1}{ 4}(1, a)$.
 
Recall that $\lambda\colon\tilde T\rightarrow S/G_0$ is the minimal resolution (cf.~\eqref{diag: resol}). Every fixed point of $\alpha$ of weight $\frac{1}{4}(1,1)$ results in a $(-4)$-curve on $\tilde T$, which should not happen by Corollary~\ref{cor: X  T min}. This implies that $k_1=0$. Combined with Lemma~\ref{lem: V=0} and \eqref{eq: equiv sign 2}, we obtain 
\begin{equation}\label{eq: H2=k3}
 H^2 = k_3\geq0.
\end{equation}

By Lemma~\ref{lem: bou sin fib} we have
\begin{equation}\label{eq: deg r 1}
\deg R=\sum_{b\in B}\deg R_b\leq\sum_{b\in B}\epsilon(F_b)=e(S),
\end{equation}
where $\epsilon(F_b)$ denotes the topological defect of the fibre $F_b$ and the last equality follows from Lemma~\ref{lem: top Euler}. By Lemma~\ref{lem: MH}, \eqref{r0} and Lemma~\ref{lem: V=0} we have
\begin{equation}\label{eq: deg r 2}
 \deg R = HK_S + H^2 =MH+ 2H^2 =4\chi(\mo_S) + 2H^2 =e(S) + 3H^2.
\end{equation}
Combining \eqref{eq: H2=k3}, \eqref{eq: deg r 1} and \eqref{eq: deg r 2} we obtain $H^2=0$. Hence  $K_S^2=8\chi(\mo_S)$ by Lemma~\ref{lem:
V=0}.

The equivalence of the two equalities of the proposition follows from the Noether formula $12\chi(\mo_S)=K_S^2 + e(S)$.
\end{proof}

\begin{proof}[Proof of Theorem~\ref{thm: isog k=1}]
By Lemma~\ref{lem: V=0} and Proposition~\ref{prop: 8X}, we have $H^2=0$ and $ K_SH=(M+H)H= MH =4\chi (\mo_S)$.
Combined with (\ref{r0}) we obtain $\deg R=4\chi (\mo_S)=e(S)$.

In view of \eqref{eq: deg r 1} the inequality in Lemma~\ref{lem: bou sin fib} becomes an equality for any point $b\in B$ and in this case  $\deg R_b = \epsilon(F_b)=2$ holds for any singular fibre $F_b=a_S^*b$. Thus the singular fibres of $a_S$ land in the list of Lemma~\ref{lem: e_f=2} for $g=3$.

Write $S^\sigma=H\cup I$ where $I$ is a finite subset of $S^\sigma$ not intersecting $H$. Then, setting $I_b:=I\cap F_b$,
\begin{equation}\label{eq: e Fix}
 e\left(S^\sigma\right)=e(H)+e(I)=\sum_{b\in B}\left(\#\,I_b - \deg R_b\right).
\end{equation}

\medskip

\noindent{\bf Claim.} For a singular fibre $F_b$ in Lemma~\ref{lem: e_f=2} it holds
\begin{equation}\label{eq: Fb}
 \#\,I_b - \deg R_b\leq \deg R_b
\end{equation}
with equality only if $F_b=2C$ with $C$ a smooth curve of genus 2.\begin{proof}[Proof of the claim]
Since $\deg R_b=2$ it is equivalent to proving 
\begin{equation}\label{eq: leq4}
  \#\,I_b\leq 4.
\end{equation}

The fibre $F_b$ is singular at the points of $I_b$ by Lemma~\ref{lem: fixed}. On the other hand, a fibre of type (ii)-(vi) in Lemma~\ref{lem: e_f=2} is reduced and has at most 2 singularties, so the strict inequality of \eqref{eq: leq4} holds. If $F_b$ is a singular fibre of type (i), i.e., $F_b=2C$ with $C$ a smooth curve of genus 2, then $\#\,F_b^\sigma\leq 6$ and hence 
\[
 \#\,I_b=\#\,F_b^{\sigma}-\#\,H\cap F_b\leq 4.
\]
\end{proof}

Plugging \eqref{eq: Fb} into \eqref{eq: e Fix} we obtain
\[
e(S)= e(S^\sigma)=\sum_{b\in B}\left(\#\,I_b - \deg R_b\right)\leq \sum_b \deg R_b = e(S).
\]
Therefore the inequality in the claim becomes an equality for any singular $F_b$ and we infer that $F_b=2C$ where $C$ is a smooth curve of genus 2.

Since $K_S^2=8\chi(\mo_S)$ by Proposition~\ref{prop: 8X}, we can conclude that $S$ is a surface isogenous to a product of unmixed type by \cite[Lemma~5]{Se95}.
\end{proof}

\section{Examples}\label{sec: examples}
In this section  we construct explicitly irregular surfaces $S$ of
general type  with $|\au(S)|=3$ and $4$. The examples of surfaces with $|\au(S)|=4$  are quite exhaustive since they include (compare Theorems~\ref{thm: G=4 k=2} and \ref{thm: bound k=1}):
\begin{itemize}
 \item surfaces with any positive geometric genus,
 \item surfaces with $g(a_S)=5$ and $\au(S)\cong(\bz/2\bz)^2$,
 \item surfaces with $g(a_S)=3$ and $\au(S)\cong(\bz/2\bz)^2$, and
 \item surfaces with $g(a_S)=3$ and $\au(S)\cong\bz/4\bz$.
\end{itemize}
For the examples of surfaces with $|\au(S)|=3$ the genus of the Albanese fibration is 4.

Examples~\ref{ex: 1}, \ref{ex: 2} and \ref{ex: Z3} take advantage of the construction of surfaces of general type with $p_g(S)=0$ in \cite{BC04}. 

In Examples~\ref{ex: 1} and \ref{ex: 2} we take the group $G\cong (\bz/2\bz)^3$ together with one of the two $G$-coverings $C\rightarrow \bar C\cong\bp^1$ in \cite[3.1]{BC04} and then construct a suitable $G$-covering $D\rightarrow \bar D$ over an elliptic curve $\bar D$. Our surfaces are then $S=(C\times D)/\Delta_G$ and $\au(S)$ turns out to be a subgroup of $(G\times G)/\Delta_G$ which has an induced action on $S$. Here $\Delta_G$ is the diagonal of $G\times G$.

Via a similar procedure, applied to the one of the two $(\bz/3\bz)^2$-coverings $C\rightarrow \bar C\cong\bp^1$ in \cite[3.3]{BC04} together with another $(\bz/3\bz)^2$-covering $D\rightarrow \bar D$ with $\bar D$ being an elliptic curve, we construct irregular surfaces with $\au(S)\cong\bz/3\bz$ in Example~\ref{ex: Z3}.

It is not clear if one can use the two equivalent $(\bz/2\bz)^4$-coverings in \cite[3.2]{BC04} to construct irregular surfaces with $\au(S)\cong(\bz/2\bz)^2$ in the same way. The $(\bz/5\bz)^2$-coverings in \cite[3.4]{BC04} do not work out, as is predicted by our bound $|\au(S)|\leq 4$. 

Example~\ref{ex: 3} with $\au(S)\cong\bz/4\bz$ does not fall into the pattern of the other examples. There the surfaces are still of the form $(C\times D)/\Delta_G$, as they should be. However, the group $\au(S)$ is not contained in $(G\times G)/\Delta_G$ any more.

The following result on the cohomology representation of the
group of automorphisms of a curve will be used in Examples \ref{ex: 1}, \ref{ex: 2} and \ref{ex: Z3}.

\begin{lem}\label{lem: non0 es}
Let $C$ be a smooth curve of genus $g(C) \geq 2$ and $G$ a finite abelian group of automorphisms of $C$. 
\begin{enumerate}
 \item[(i)] Assume $g(C/G) =1$. Then, for any $\chi\in\widehat G$, $H^1(C,
 \bc)^\chi \neq 0$ if and only if $\chi(\sigma)\neq 1$ for some stabilizer $\langle\sigma\rangle$ over a point of $C/G$.
 \item[(ii)] Assume $g(C/G)=0$. Then, for any $\chi\in\widehat G$, $H^1(C,
 \bc)^\chi \neq 0$ if and only if there are stabilizers $\langle\sigma_1\rangle,\langle\sigma_2\rangle,\langle\sigma_3\rangle$ over 3 distinct points of $C/G$ such that $\chi(\sigma_i)\neq 1$ for $1\leq i\leq 3$.
\end{enumerate}
\end{lem}
\begin{proof}
 This is a consequence of \cite[Proposition~2]{B87} or \cite[p.~244]{B91}.
\end{proof}

\subsection{Examples of irregular surfaces with $|\au(S)|=4$}
\begin{ex}[$\au(S)\cong(\bz/2\bz)^2$ and $g(a_S)=5$]\label{ex: 1}
We take the group $G=\langle e_1,e_2,e_3\rangle\cong\mathbb (\bz/2\bz)^3$. Let $\bar C,\bar D$ be two smooth curves of genera $g(\bar C)=0,\,g(\bar D)=1$ respectively. 

By Riemann's existence theorem there is a $G$-covering
$C\rightarrow \bar C$ with 6 branch points, over which the stabilizers are $ \langle e_1\rangle,\langle e_1\rangle,\langle e_2\rangle,\langle e_2\rangle, \langle e_3\rangle,\langle e_3\rangle$ respectively. Similarly, there is a $G$-covering $D\rightarrow \bar D$ with $2r$ branch points, over which the stabilizers are all $\langle e_1+e_2+e_3 \rangle$.

Consider the product action of $G\times G$ on $C\times D$. Since
\[
 \langle e_1\rangle\cap \langle e_1+e_2+e_3\rangle=\langle e_2\rangle\cap \langle e_1+e_2+e_3 \rangle=\langle e_3\rangle\cap \langle e_1+e_2+e_3\rangle =\{0\},
\]
the induced action of the diagonal subgroup $\Delta_G\subset G\times G$ on $C\times D$ is free. Therefore $S:=(C\times
D)/\Delta_G$ is a surface isogenous to a product of unmixed type. One calculates easily $ g(C)=5 \text{ and } g(D)=4r+1$ by Hurwitz's formula. So 
\[
 K_S^2=\frac{8}{|G|}(g(C)-1)(g(D)-1)=16r\text{ and } \chi(\mo_S) =\frac{1}{8}K_S^2 = 2r
\]
and our surfaces form an infinite series as $r$ varies. The irregularity of $S$ is $q(S)=g(\bar C)+g(\bar D)=1$. The Albanese map of $S$ is the induced fibration $S\rightarrow \bar D$ and has fibre genus $g(C)=5$.

Consider the character $\chi$ of $G$ such that $\chi(e_1)=\chi(e_2)=\chi(e_3)=-1$. By Lemma~\ref{lem: non0 es} this is the only character
$\chi$ satisfying the following conditions:
\[
H^1(C,\mathbb C)^\chi\neq 0 \text{ and }H^1(D,\mathbb
C)^{\bar\chi}\neq 0.
\]
Then, by the expression of $H^2(S,\bc)$ in \cite[(4.5.2)]{CLZ13}, $\ker(\chi)$ acts trivially on $H^2(S,\bc)$. One also sees easily that $\ker(\chi)$ acts trivially on $H^1(S,\bc)$, so it is in fact a subgroup of $\au(S)$.

Now we calculate: $\ker(\chi)=\langle e_1+e_2,e_1+e_3\rangle\cong(\bz/2\bz)^2$. On the other hand,  it holds $|\au(S)|\leq 4$ by Theorem~\ref{main}. Hence 
\[
 \au(S)=\ker(\chi)\cong (\bz/2\bz)^2.
\]
\end{ex}
\begin{rmk}
As is pointed out by a referee, the curve $C$ in Example~\ref{ex: 1} is the so-called Kummer covering of the rational curve $\bar C$ of type $(2,2,2)$, defined by the homogeneous equations 
\[
z_1^2=Q_1(x,y),\,z_2^2=Q_2(x,y),\,z_3^2=Q_3(x,y)
 \]
where $Q_i(x,y)$ are quadratic polynomials in $x,y$ for $1\leq i\leq 3$. Its quotient by $\ker(\chi)$ is the genus 2 curve defined by the weighted homogeneous equation 
\[
z^2=Q_1(x,y)Q_2(x,y)Q_3(x,y). 
\]
The other curve $D$ is the normalization of the fibre product $D_1
\times_{\bar D} D_2$ where $D_1\rightarrow \bar D$ is an isogeny of elliptic curves of degree 2 and $D_2\rightarrow \bar D$ is a double covering with the same branch locus as $D\rightarrow \bar D$.
\end{rmk}

\begin{ex}[$\au(S)\cong(\bz/2\bz)^2$ and $g(a_S)=3$]\label{ex: 2} The construction is similar to Example~\ref{ex: 1}. Let $G=\langle e_1,e_2,e_3\rangle\cong(\bz/2\bz)^3$, and let $\bar C,\ \bar D$ be two smooth curves of genera $g(\bar C)=0,\,g(\bar D)=1$ respectively. 

We can construct by the Riemann existence theorem a $G$-covering $C\rightarrow \bar C$ with 5 branch points, over which the stabilizers are $\langle e_1\rangle,\langle e_1\rangle,\langle e_2\rangle,\,\langle e_3\rangle,\,\langle e_2+e_3\rangle$ respectively and another $G$-covering $D\rightarrow \bar D$ with $2r$ branch points,  over which the stabilizers are all $\langle e_1+e_3\rangle$.

Consider the product action of $G\times G$ on $C\times D$. Since
\[
 \langle e_1+e_3\rangle\cap \langle e_1\rangle=\langle e_1+e_3\rangle\cap \langle e_2\rangle=\langle e_1+e_3\rangle\cap \langle e_3\rangle = \langle e_1+e_3\rangle\cap \langle e_2+e_3\rangle =\{0\},
\]
the induced action of the diagonal subgroup $\Delta_G\subset G\times G$ on $C\times D$ is free, and hence $S:=(C\times
D)/\Delta_G$ is a surface isogenous to a product of unmixed type. By Hurwitz's formula one computes $g(C)=3$ and $g(D)=4r+1$. So
\[
 K_S^2=\frac{8}{|G|}(g(C)-1)(g(D)-1)=8r \text{ and }\chi(\mo_S) =\frac{1}{8}K_S^2= r.
\]
The irregularity of $S$ is $q(S)=g(\bar C)+g(\bar D)=1$. The Albanese map $a_S$ is the induced fibration $S\rightarrow \bar D$ and hence has fibre genus $g(C)=3$.

The character $\chi$ of $G$ with $\chi(e_1)=\chi(e_2)=\chi(e_1+ e_3)=-1$ is the only one satisfying the following conditions:
\[
H^1(C,\mathbb C)^\chi\neq 0 \text{ and }H^1(D,\mathbb
C)^{\bar\chi}\neq 0.
\]
Using the same argument as in Example~\ref{ex: 1} we infer that
\[
\au(S)=\ker(\chi)\cong (\bz/2\bz)^2.
\]
\end{ex}

\begin{rmk}
The genus 3 curve $C$  in Example~\ref{ex: 2} is hyperelliptic by Lemma~\ref{lem: g3}. A referee writes down its affine equation as follows:
\[
y^2= (x^4+ax^2+1)(x^4+bx^2+1) \text{ with } a,b\in\bc\setminus\{\pm 2\}.
\]
\end{rmk}

In  Examples \ref{ex: 1} and \ref{ex: 2}, the group $\au(S)$ is contained in $(G\times G)/\Delta_G$, viewed as a subgroup of $\aut(S)$. But this is not
the case in the following example.

\begin{ex}[$\au(S)\cong\bz/4\bz$ and $g(a_S)=3$]\label{ex: 3} 
This time take the group $G = \left<e_1,e_2\right>\cong(\bz/2\bz)^2$. Write $e_3:=e_1+e_2$. For $j=1$, $2$, $3$, let $\chi_j$ be the character of $G$  with $\ker(\chi_j)=\left<e_j \right>$ and $\chi_0$ the character of the principle representation. For any $0\leq j\leq 3$, since $\chi_j$ takes values in $\{1,-1\}$, we have $\chi_j=\bar\chi_j$. 

Let $C$ be a hyperelliptic curve whose affine equation is
$$y^2=(x^4+1)(x^4+a),\ \ \ a\in \mathbb{C}\setminus\{0, 1\}.$$
The hyperelliptic involution $\tau$ acts by $(x,y)\mapsto
(x,-y)$. There is another automorphism $\gamma$ of $C$ given by $(x,y)\mapsto
(\sqrt{-1}x,y)$.  The 1-forms $\omega_j:=\frac{x^jdx}{y}$ ($j=0, \ 1,\
2$) constitute a basis of $H^0(C, \Omega_C^1)$ and we have $\gamma^*\omega_j=\sqrt{-1}^{j+1}\omega_j$.

There is an action of $G$  on $C$ such that $e_1$ acts as $\tau$ and $e_2$ acts as $\gamma^2$. It is easy to see that $e_3$ acts freely on $C$, so $g(C/e_3)=2$ by the Riemann--Hurwitz formula. Moreover, $g(C/e_2)=1$ and $g(C/e_1)=0$. 

The nonzero eigenspaces of the $G$-action on $H^1(C,\bc)$ are as follows:
\begin{equation}\label{eq: eigen C}
 H^1(C,\bc)^{\chi_3}=\bigoplus_{j=2,3}\left(\bc\omega_j\oplus\bc \bar\omega_j\right), \,H^1(C,\bc)^{\chi_2}=\bc\omega_1\oplus\bc \bar\omega_1.
\end{equation}

Now let $\bar D$ be an elliptic curve and  $\delta_1$ and
$\delta_2$ two non-isomorphic  invertible sheaves of degree $r$
($r>0$) such that $\delta_1^{\otimes2}\sim\delta_2^{\otimes2}$.
Let $B\in |\delta_1^{\otimes2}|$ be a reduced divisor and $\pi_i
\colon D_i\to \bar D$ the double cover defined by the data $(B, \delta_i)$.  We
have a commutative diagram
 \begin{center}
  \begin{tikzpicture}
      \node (D) at (-3,2) {$D$};
      \node (product) at (0,0) {$D_1\times_{\bar D} D_2$};
      \node (D2) at (2,0) {$D_2$};
      \node (D1) at (0,-1.8) {$D_1$};
      \node (DD) at (2,-1.8) {$\bar D$};
      \draw[->] (D)--(product)  node[inner sep=2pt,fill=white, midway]{$\mu$};;
      \draw[->] (D)--(D1) node[below, midway]{$\mu_1$};
      \draw[->] (D)--(D2) node[above, midway]{$\mu_2$};
      \draw[->] (product)--(D1);
      \draw[->] (product)--(D2);
      \draw[->] (D1)--(DD) node[below, midway]{$\pi_1$};
      \draw[->] (D2)--(DD)node[right, midway]{$\pi_2$};
  \end{tikzpicture}
  \end{center}
where $\mu$ is the normalization morphism.

For $i=1$, $2$ let $\beta_i$  be the involution of $D$
corresponding to the double cover $\mu_i$, and write
$\beta_3=\beta_1\beta_2$.  Then $D$ is a curve of
genus $g(D)=2r+1$ and there is an action of $G$ on $D$ such that $e_i$ acts as $\beta_i$, $i=1,2$. By construction $\beta_1$ and $\beta_2$ act freely on $D$.

We have $H^1(D, \bc)^{\chi_0}= (\pi_1\circ\mu_1)^*H^1(\bar D, \bc)$ and 
\begin{equation}\label{eq: eigen D}
  \mu_j^*H^1(D_j, \bc)=H^1(D,\bc)^{\chi_0}\oplus H^1(D,\bc)^{\chi_j} \text{ for }j=1, 2.
\end{equation}
Combining these with the equality $\sum_{\chi\in
\widehat G}\dim_\bc H^1(D,\bc)^{\chi}=\dim_\bc H^1(D,\bc)$, we have $ H^1(D,\bc)^{\chi_3}=0$.

Consider the product action of $G\times G$ on $C\times D$. The induced action of diagonal subgroup $\Delta_G$ on $C\times D$ is free. So the quotient $S=(C\times D)/\Delta_G$ is a surface isogenous to a product of unmixed type, whose invariants are
\begin{align*}
\textrm{$p_g(S)=r$, $q(S)=1$ and $K_S^2=8r$.}
\end{align*}
By the calculation of the eigenspaces of the $G$-actions on the cohomology groups $H^1(C,\bc)$ and $H^1(D,\bc)$ as in \eqref{eq: eigen C} and \eqref{eq: eigen D} respectively, we infer that (\cite[(4.5.2)]{CLZ13})
 \begin{equation}\label{eq: 4.1}
 \begin{split}H^2(S, \bc)&=W\bigoplus
 \left(\bigoplus_{\chi\in \widehat G}
  H^1(C,\bc)^\chi\otimes H^1(D,\bc)^{\bar\chi}\right) \\
 & =W\bigoplus
H^1(C,\bc)^{\chi_2}\otimes H^1(D,\bc)^{\chi_2},\end{split}\end{equation}
where $W =
 H^0(C,\bc)\otimes H^2(D,\bc)\bigoplus H^2(C,\bc)\otimes H^0(D,\bc)$.

Let $\alpha$ be the automorphism  of $S$ induced by $\gamma\times\beta_3\in \op{Aut}(C\times D)$. Then
 $\alpha$ is of order 4. Note that $\gamma$ and $\beta_3$ acts as $-\id$ on $H^1(C,\bc)^{\chi_2}$ and $H^1(D,\bc)^{\chi_2}$  respectively. Hence $\gamma\times\beta_3$ acts trivially on the right hand side of \eqref{eq: 4.1}. It follows that the action of $\alpha$ on $H^2(S, \bc)$ is trivial. Of course, $\alpha$ acts trivially on $H^1(S,\bc)=a_S^*H^1(\bar D,\bc)$, where $a_S\colon S\rightarrow \bar D$ is the Albanese map. Hence $\alpha$ is in $\au(S)$. Since $|\au(S)|\leq 4$ by Theorem~\ref{main}, it can only happen that 
\[
      \au(S)=\langle\alpha\rangle\cong\bz/4\bz.                                                                                                                                                                                                                                                                                                                                                                        
\]
\end{ex}
\begin{rmk}
(i) Example~\ref{ex: 2} (resp.~Example~\ref{ex: 3}) exhausts the possible values of $\chi(\mo_S)$ and hence also of $K_S^2,\,p_g(S),\,e(S)$ of irregular surfaces $S$ with $q(S)=1$, $g(a_S) =3$ and $\au(S)\cong(\bz/2\bz)^2$ (resp.~$\au(S) \cong \bz/4\bz$).

(ii) For all of surfaces $S=(C\times D)/\Delta_G$ in the above examples the group $(G\times G)/\Delta_G\subset\aut(S)$ is not contained in $\au(S)$. In fact, this is true more generally. Namely, let $S=(C\times D)/\Delta_G$ be a surface isogenous to a product of unmixed type with $G$ \emph{abelian} and $q(S)=1$. Then the coset $(G\times G)/\Delta_G$ is well-defined as a group and has an induced action on $S$. The quotient of $S$ by $(G\times G)/\Delta_G$ is isomorphic to $(C/G)\times (D/G)$ with $g(C/G) +g(D/G)=1$ and hence has geometric genus 0. This implies that $(G\times G)/\Delta_G$ is not contained in $\au(S)$.

This phenomenon is reflected in the fact that the smooth fibres of the Albanese map of surfaces of general type with $q(S)=1$ and $|\au(S)|=4$ have an extra involution, see Appendix~\ref{sec: extra}.
\end{rmk}

All in all a remaining problem is to classify irregular surfaces of general type with $|\au(S)|=4$.

\subsection{Examples of irregular surfaces with $\au(S)\cong\bz/3\bz$}
\begin{ex}[$\au(S)\cong\bz/3\bz$ and $g(a_S)=4$]\label{ex: Z3}
Let $G=\langle e_1,e_2\rangle$ be a finite group isomorphic to $(\bz/3\bz)^2$. By the Riemann existence theorem one can construct a $G$-covering $C\rightarrow \bp^1$ with 4 branch points, over which the stabilizers are generated by $e_1, e_2, 2e_1, 2e_2$ respectively. By the Riemann--Hurwitz formula we have $g(C) = 4$. Similarly, one constructs another $G$-covering $D\rightarrow \bar D$ over an elliptic curve $\bar D$ such that there are $3r$ branch points, over which the stabilizers are all generated by $2e_1+2e_2$.  The genus $g(D)$ is $9r+1$ by the Riemann--Hurwitz formula.

Since the two systems of stabilizers of the coverings $C\rightarrow \bp^1$ and $D\rightarrow \bar D$ as above are disjoint, the induced action of the diagonal subgroup $\Delta_G\subset G\times G$ on $C\times D$ is free. Therefore $S:=(C\times D)/\Delta_G$ is a surface isogenous to a product of unmixed type with invariants 
\[
 K_S^2=\frac{8}{|G|}(g(C)-1)(g(D)-1)= 24r\text{ and } \chi(\mo_S) =\frac{1}{8}K_S^2 = 3r.
\]
We have $q(S)=g(\bar D)=1$. The surfaces form an infinite series as $r$ varies. 

Consider the character $\chi$ such that
\[
\chi(e_1)=\chi(e_2) = \exp(\frac{2\pi\sqrt{-1}}{3}).
\]
By Lemma~\ref{lem: non0 es}, $\chi$ and $\chi^2$ are the only characters
whose eigenspaces of the $G$-actions on $H^1(C,\bc)$ and $H^1(D,\bc)$ are simultaneously nonzero. By the expression of $H^2(S,\bc)$ in \cite[(4.5.2)]{CLZ13}, $\ker(\chi)\cong\bz/3\bz$ acts trivially on $H^2(S,\bc)$. One also sees easily that $\ker(\chi)$ acts trivially on $H^1(S,\bc)$, so it is in fact a subgroup of $\au(S)$.

Since $|\au(S)|\leq 4$ by Theorem~\ref{main}, it must hold
\[
 \au(S)=\ker(\chi)\cong \bz/3\bz. 
\]
\end{ex}

\appendix

 \section{Topological defect of curves}\label{sec: top
 def} 
 \begin{defn}\label{def: top def}
 For any effective divisor $D$ on a smooth projective surface $S$ we define
 $$\epsilon(D) = e(D) + 2p_a(D) - 2.$$  It is called the
 \emph{topological defect} of $D$.
 \end{defn}
 It is well-known that $\epsilon(D)\geq 0$ and the equality holds if
 and only if  $D$ is a smooth curve. If $D$ is reduced then
 $\epsilon(D)$ is the sum of local contributions from the
 singularities.
  \begin{lem}\label{lem: delta  reduced}
  Let $D\subset S$ be a reduced curve on a smooth projective surface. For a point
  $p\in D$ denote by $\mu_p(D)$ the multiplicity of $D$ at $p$. Then
  $$\epsilon(D)\geq \sum_{p\in D} (\mu_p(D)-1)^2,$$ and the equality holds if and only
  if every singularity $p$ of $D$ is ordinary, that is, the strict
  transform of $D$ in the blow-up of $S$ at every singularity $p$ of
  $D$ contains exactly $\mu_p(D)$ points over $p$.
  \end{lem}
  \begin{proof}
  Let $\rho\colon\tilde S\rightarrow S$ be the simultaneous blow-up of
  $S$ at all the singularities of $D$ and $E_p$ the exceptional
  divisor over a singularity $p$ of $D$. Let $\tilde D\subset \tilde
  S$ the strict transform of $D$. Then $\tilde D= \rho^*D -\sum_p
  \mu_p(D) E_p$. As a set the inverse image of $p$ in $\tilde
  D$ is $E_p\cap\tilde D $ and $e(\tilde
  D)=e(D) + \sum_p(\#\,\tilde D \cap E_p - 1)$. We have
  $\#\,E_p\cap\tilde D\leq \mu_p(D)=\tilde D  E_p$ and  the equality
  holds  if and only if $p\in D$ is an ordinary singularity. If $p\in
  D$ is an ordinary singularity then $\tilde D$ is already smooth over
  $p$.
  
  Now we have
  \begin{align*}
   \epsilon(D) &= e(D) + (K_S + D) D\\
             &= e(D) + \left(\rho^* K_S + \rho^* D\right) \rho^* D\\
             &= e(D) +\left(K_{\tilde S} + \tilde D + \sum_p(\mu_p(D)-1) E_p\right) \left(\tilde D +\sum_p \mu_p(D) E_p\right)\\
             &= e(\tilde D)+ 2p_a(\tilde D) - 2 - \sum_p \left(\#\,\tilde D \cap E_p - 1\right) +\sum_p\mu_p(D)(\mu_p(D)-1)\\
             &\geq e(\tilde D)+ 2p_a(\tilde D) - 2- \sum_p(\mu_p(D) -1) +\sum_p\mu_p(D)(\mu_p(D)-1)\\
             &= \epsilon(\tilde D)  + \sum_p (\mu_p(D) -1)^2\\
             &\geq \sum_p (\mu_p(D) -1)^2.
  \end{align*}
   If $\epsilon(D)=\sum_p
  (\mu_p(D) -1)^2$ then both of the inequalities above
  become equalities and this is equivalent to each singularity of $D$
  being ordinary.
  \end{proof}

  In case $D$ is nonreduced the situation is more complicated.
  Nevertheless we will try to get a control on $\epsilon(D)$ when $D$ is
  a fibre of some fibration. From the following lemma we see that the
  topological defect of a fibre has contributions from the
  singularities of the reduced part as well as the irreducible
  components with multiplicity.
  \begin{lem}\label{lem: delta Fb}
  Let $f\colon S\rightarrow B$ be a fibration of a smooth projective surface onto
  a curve and $F$ a fibre of $f$. Write $F_\red$ for the
  reduced part of $F$. Then we have
  $$
   \epsilon(F)
             \geq \epsilon(F_\red) +K_S(F - F_\red),
  $$
  where the equality holds  if and only if  $F= m F_\red$ for a  positive integer $m$.
  \end{lem}
  \begin{proof}
  We have
  \begin{align*}
   \epsilon(F) & = e(F_\red) + 2p_a(F_\red) -2 + 2p_a(F) - 2p_a(F_\red)\\
                 &= \epsilon(F_\red) + 2p_a(F) - 2p_a(F_\red)\\
                 & = \epsilon(F_\red) + K_S(F - F_\red) + F^2-F_\red^2\\
                & = \epsilon(F_\red)+ K_S(F - F_\red) -F_\red^2\hspace{2cm} (\text{since } F^2=0)\notag\\
                &\geq \epsilon(F_\red)+ K_S(F - F_\red).
                \end{align*}
and the inequality becomes equality if and only if $F_\red^2=0$.
   By Zariski's lemma (\cite[III, Lemma~8.2]{BHPV04}) the later
   is equivalent to $F= m F_\red$
  for some positive integer $m$.
  \end{proof}

 The usefulness of topological defects of singular fibres lies in the
 fact that they determine the (global) topological Euler
 characteristic of the fibration.
 \begin{lem}[{\cite[III, Prop.~11.4]{BHPV04}}]\label{lem: top Euler}
  Let $f\colon S\rightarrow B$ be a fibration from a smooth projective surface onto a smooth curve $B$. Let $F$ be a smooth fibre of $f$ and $F_b$ a fibre over any point $b\in B$. Then 
\[
   e(S) = e(F)e(B)+ \sum_{b\in B} \epsilon(F_b). 
\]
In particular, if the genus $g(B)=1$ then $e(S)=\sum_{b\in B}\epsilon(F_b)$.
 \end{lem}
 
For the reader's convenience, we recall the classification of singular fibres with topological defects 1 and 2 obtained in \cite{Cai01}.
   \begin{lem}\label{lem: e_f=1}\cite[Remark 2.6]{Cai01} 
   Let  $f \colon S \rightarrow B$ be a relatively minimal
   fibration of genus $g\geq 3$, and  $F_b$  a singular fibre of  $f$. If
   $\epsilon(F_b)=1$ (cf. Definition A.1)
  then $F_b$ belongs to  one of the following types.
   \begin{enumerate}
   \item
     an
   irreducible  curve  with exactly one node;
   \item
     a  sum of two smooth irreducible  curves meeting
   transversally in a point.
   \end{enumerate}
   \end{lem}
  \begin{lem}\label{lem: e_f=2}\cite[Lemma 2.5]{Cai01} Let  $f $ and
   $F_b$  be as in Lemma~\ref{lem: e_f=1}.
  If  $\epsilon(F_b)=2$ then $F_b$ belongs to  one of the following
  types.
   \begin{enumerate}
   \item
   $F_b=2C$, where $C$ is an irreducible smooth curve of genus $2$
  (this case occurs only when $g=3$);
   \item
   $F_b$ is an irreducible  curve  with exactly two nodes, and the
  normalization of $F_b$ is a curve of genus $g-2$;
  \item
   $F_b$ is an irreducible  curve  with one cusp, and the
   normalization of $F_b$ is a curve of genus $g-1$;
   \item
    $F_b=C_1+C_2$, where $C_i$ are irreducible
    curves meeting
   transversally in a point, and either $C_1$ or $C_2$ (and not both)
   has a node;
  \item
   $F_b=C_1+C_2$, where $C_i$ are irreducible smooth curves
  meeting transversally in two points, and $g(C_1)+g(C_2)=g-1$;
  \item $F_b=C_1+C_2+C_3$, where $C_i$ are irreducible smooth
   curves with $C_1C_2=C_2C_3=1$, $C_1C_3=0$,
    and $g(C_1)+g(C_2)+g(C_3)=g$.
   \end{enumerate}
  \end{lem}

We have the following description of a fibre containing an isolated fixed point of an automorphism acting on a fibration.
\begin{lem}(\cite[Lemma 1.4]{Cai09}, \cite[Lemma 2.2]{Cai12})\label{lem: fixed}
Let $f\colon S\to B$ be a relatively minimal fibration of genus $g\geq1$, and $\sigma$  an automorphism  of finite  order $r$ of $S$ with $f\circ\sigma=f$. Let $p\in S$ be an isolated fixed point of $\sigma$ and $F_b$ the fibre containing it.  Then the following holds.
\begin{enumerate}
  \item $F_b$ is singular  at $p$;
\item if  the multiplicity $\textrm{mult}_pF_b=2$ and $r$ is an odd prime, then $p$ is a node of $F_b$, and  the action of $\sigma $ at    $p$ is of weight ${1\over r}(1, r-1)$;
\item if the action of $\sigma$ at  $p$ is of weight ${1\over 4}(1, 1)$, then $\textrm{mult}_pF_b$ is divisible by $4$;
\item if the action of $\sigma$ at  $p$ is of weight ${1\over 4}(1, 2)$ {\rm (}resp. ${1\over 4}(1, 3)${\rm )} and $\op{mult}_pF_b=2$, then $p\in F_b$ is not {\rm (}resp. is{\rm )}  an ordinary node.
\end{enumerate}
\end{lem}

\section{An extra involution of curves of genera 3 and 5}\label{sec: extra}
Let $C$ be a smooth projective curve and $G$ a finite group of automorphisms. In certain situations one can lift an automorphism of the quotient $C/G$ to $C$. This is the case when $C$ is a smooth fibre of the Albanese fibration $a_S\colon S\rightarrow \alb(S)$ for a surface of general type with $q(S)=1$ and $|\au(S)|=4$ and the  group $G$ is the restriction of $\au(S)$ to $C$.

In disguise the following result is contained in \cite[Theorem~3.4]{Pol06}. 
\begin{lem}\label{lem: g3}
Let $C$ be a smooth curve of genus 3. Suppose that a group $G\cong (\bz/2\bz)^2$ acts faithfully on $C$ with $g(C/G)=1$. Then $C$ is hyperelliptic.
\end{lem}
\begin{proof}
Let $\sigma_1,\sigma_2$ and $\sigma_3$ be the three involutions from $G$. Then we have $\langle\sigma_i\rangle\cap\langle\sigma_j\rangle=\{\id\}$ for $1\leq i<j\leq 3$ and $G=\bigcup_{1\leq i\leq 3}\left<\sigma_i\right>$. The following holds by \cite[Thm.~5.9]{Ac94}:
\[
2g(C) + 4 = \sum_{1\leq i\leq 3} 2g(C/\sigma_i).
\]
Necessarily there is an $i$ such that $g(C/\sigma_i)=2$.  The curve $C$ is hyperelliptic by \cite[Lemma~5.10]{Ac94}.
\end{proof}

\begin{rmk}
Let $C$ be a smooth curve of genus 3. Suppose that $G\cong \bz/4\bz$ acts faithfully on $C$ with $g(C/G)=1$. Then $C$ also has an additional involution $\tau$ such that the group of automorphisms generated by $\tau$ and $G$ is isomorphisc to $\bz/4\bz\oplus\bz/2\bz$ if $C$ is hyperelliptic and is isomorphic to $D_8$ if $C$ is not hyperelliptic (\cite[Sec.~6.6.5]{Dolg12}). 

We point out here an error in the classification of the full automorphism groups of curves of genus 3 in \cite{KK79}, where curves $C$ of genus 3 with $\aut(C)\cong\bz/4\bz$ and $g(C/\aut(C))=1$ are allowed, see \cite[page~295]{KK79}.
\end{rmk}

\begin{lem}\label{lem: g5}
Let $C$ be a smooth curve of genus 5 and $G$ a group of order 4 acting freely on $C$. Then the quotient $C/G$ is a curve of genus 2 and the hyperelliptic involution of $C/G$ lifts to an involution of $C$.
\end{lem}
\begin{proof}
The assertion that $g(C/G)=2$ follows from the Riemann--Hurwitz formula. The hyperelliptic involution of $C/G$ lifts to an involution of $C$ by \cite[Cor.~4.13 and 4.12]{Ac94}.
\end{proof}

\begin{cor}\label{cor: g5}
Let $f\colon S\rightarrow B$ be a relatively minimal fibration of genus $5$ from a smooth projective surface $S$ onto a smooth projective curve $B$. Suppose that $G\subset\aut(S)$ is a finite group of automorphisms of order 4, which preserves the fibres of $f$ and acts freely on the smooth fibres. Then there is an involution $\tau\in\aut(S)\setminus G$ preserving the fibres of $f$.
\end{cor}
\begin{proof}
Let $U\subset B$ be an open subscheme, over which the fibration $f$ is smooth. Denote $S_U=f^{-1}(U)$. Then the group $G$ acts freely on $S_U$ and, by Lemma~\ref{lem: g5}, the quotient fibration $h\colon S_U/G\rightarrow U$ has genus 2. Moreover, for any $b\in U$, the hyperelliptic involution $\bar\tau_b$ of  the genus 2 curve $f^*b/G$ lifts to an involution $\tau_b$ of $f^*b$. 

Look at the following commutative diagram of fundamental groups with exact rows:
 \begin{equation}\label{diag: fun gp}
\begin{aligned}
   \xymatrix{
  1 \ar[r] & \pi_1(f^*b, x)\ar[r]\ar@{_{(}->}[d] & \pi_1(S_U,x) \ar[r]\ar@{_{(}->}[d] &\pi_1(U,b) \ar[r]\ar@{=}[d] & 1\\
  1 \ar[r] & \pi_1(h^*b,\bar x)\ar[r] & \pi_1(S_U/G,\bar x) \ar[r] &\pi_1(U, b) \ar[r] & 1.
  }
  \end{aligned}
 \end{equation}
where $\bar x$ is a fixed point of the hyperelliptic involution $\bar\tau_b$ and $x\in f^*b$ is a point over $\bar x$. The fact that the hyperelliptic involution $\bar\tau_b$ of $h^*b$ lifts to $f^*b$ means that $\bar\tau_{b*}$ preserves the image of $\pi_1(f^*b,x)\rightarrow \pi_1(h^*b,\bar x)$. By the commutative diagram \eqref{diag: fun gp} we infer that $\bar\tau_*$ preserves the image of $\pi_1(S_U,x)\rightarrow \pi_1(S_U/G,\bar x)$, so the hyperelliptic involution $\bar\tau_U$ of the fibration $h\colon S_U/G\rightarrow U$ lifts to an automorphism $\tau_U$ of $S_U$, which preserves the fibres and induces the involution $\tau_b$ on $f^*b$. One sees immediately $\tau_U$ is an involution itself. 

The relative minimality of $f\colon S\rightarrow B$ guarantees that the involution $\tau_U$ extends to the whole $S$, still preserving the fibres.
\end{proof}

\end{document}